\documentclass[11pt]{article}

\usepackage{listings}
\usepackage[colorlinks=false]{hyperref}

\usepackage{amsmath,amsthm,amsfonts,amssymb,amscd}
\usepackage{float}
\usepackage{graphicx}

\usepackage[mathscr]{euscript}
\usepackage{stmaryrd}
\usepackage{microtype}
\usepackage{booktabs}
\usepackage{cleveref}
\usepackage{bookmark}
\usepackage{mathrsfs}

\usepackage{emptypage}
\usepackage{tikz}
\usepackage{subfigure}
\usepackage[utf8]{inputenc}
\usepackage[T1]{fontenc}

\DeclareMathAlphabet{\mathpzc}{OT1}{pzc}{m}{it}

\theoremstyle{plain}
\newtheorem{theorem}{\scshape Theorem}

\newtheorem{lemma}[theorem]{\scshape Lemma}

\theoremstyle{definition}

\newtheorem{remark}{\scshape Remark}

\newcommand{\jdel}{\mathcal{J}_\delta}

\theoremstyle{definition}

\DeclareMathAlphabet{\mathpzc}{OT1}{pzc}{m}{it}

\numberwithin{equation}{section}

\def\RR{{\mathbb R}}
\def\TT{{\mathbb S}^1}
\def\refdom{{\mathcal D}}

\def\div{{\operatorname{div}}}
\def\curl{{\operatorname{curl}}}

\def\p{{\partial\hspace{1pt}}}

\def\jump#1{{[\hspace{-2pt}[#1]\hspace{-2pt}]}}

\def\Bigjump#1{{\Big[\hspace{-4.5pt}\Big[#1\Big]\hspace{-4.5pt}\Big]}}
\def\({{(\hspace{-2pt}(}}
\def\){{)\hspace{-2pt})}}

\usepackage{fancyhdr}
\title{Well-posedness and decay to equilibrium for the Muskat problem with discontinuous permeability}

\author{Rafael Granero-Belinch\'on
\\{\small Univ Lyon, Universit\'e Claude Bernard Lyon 1}
\\{\small CNRS UMR 5208, Institut Camille Jordan}
\\{\small 43 blvd. du 11 novembre 1918}
\\{\small F-69622 Villeurbanne cedex, France}
\\{\footnotesize email: granero@math.univ-lyon1.fr}
\and
 Steve Shkoller
\\Department of Mathematics
\\University of California
\\Davis, CA 95616 USA
\\{\footnotesize email: shkoller@math.ucdavis.edu}
}
\date{\today}

\begin{document}

\maketitle

We first prove local-in-time well-posedness  for the  Muskat problem, modeling fluid flow in a two-dimensional inhomogeneous porous media. 
The permeability  of the porous medium is described by a step function,  with a jump
discontinuity across the fixed-in-time curve $(x_1,-1+f(x_1))$, while the interface separating the fluid from the vacuum region is given 
by the time-dependent curve $(x_1,h(x_1,t))$. 
Our estimates are based on a new methodology that relies upon a careful study of the PDE system, coupling Darcy's law and 
incompressibility of the fluid, rather than the analysis of the singular integral contour equation for the interface function $h$.  We are able to 
develop an existence theory for  any initial interface given by  $h_0 \in H^2$ and any permeability curve-of-discontinuity that is given by $f \in H^{2.5}$. In particular, our method allows for both curves to have (pointwise) unbounded curvature. In the case that  the permeability discontinuity
is the set $f=0$, we prove global existence and decay to equilibrium for small initial data. This decay is obtained using a new energy-energy dissipation inequality that couples tangential derivatives of the velocity in the bulk of the fluid with the curvature of the interface. To the best of our knowledge, this is the first global existence result for the Muskat problem with discontinuous permeability.

 \vspace{.2 in}
 
 \tableofcontents

\section{Introduction}
\subsection{The Muskat problem}
The Muskat problem, introduced in \cite{Muskat}, models the dynamics of an evolving material interface separating two fluids flowing through a porous medium, \emph{i.e.} a medium consisting of a solid matrix with fluid-filled pores. Porous media flow is modelled by Darcy's law
\begin{equation}\label{basic-model}
\frac{\mu}{\beta}u=-\nabla p-(0,g\rho)^T,
\end{equation} 
where $\mu$ is the viscosity of the fluid,  $\rho$ denotes the density,  $u$ is the incompressible fluid velocity, and  $p$ is  the pressure function;
additionally, 
$\beta>0$ denotes the permeability of the solid matrix, and $g$ is the acceleration due to gravity, which we shall henceforth set to $1$.   Darcy's law
\eqref{basic-model} is an {\it empirical} relation between momentum and force  (see, for example, \cite{bear, NB}), and replaces conservation of 
momentum, which is used to model the evolution of inviscid fluid flows.

The purpose of this paper is to study the evolution of an interface moving through porous media with a discontinuous permeability. As the permeability takes two different values, this case is known in the literature as the \emph{inhomogeneous Muskat problem}. Specifically, we are interested in the well-posedness and decay to equilibrium for the inhomogeneous Muskat problem.

\begin{figure}[htbp]
\centering
\includegraphics[scale=0.3]{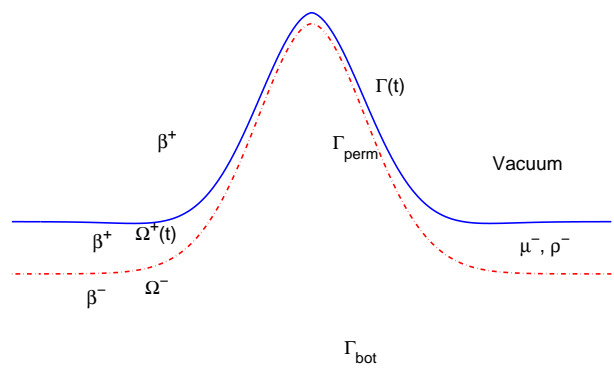}
\vspace{-.6 in}
\caption{{\footnotesize The solid curve (blue) is the interface $\Gamma(t)$ and the dashed curve (red) denotes the interface
$\Gamma_{\operatorname{perm}}$, across which  the permeability is discontinuous.}}
\label{fig1}
\end{figure} 
We let $\TT$ denote the circle, so that functions $h: \mathbb{S}  ^1 \to \mathbb{R}$ are identified with $[-\pi, \pi]$-periodic functions on $ \mathbb{R} $.
As shown in Figure \ref{fig1}, 
we consider a porous medium occupying an open time-dependent subset $\Omega(t)  \subset \TT \times \mathbb{R}   $ such that
$$\Omega(t)={\Omega^+(t)}\cup{\Omega^-} \cup \Gamma_{\operatorname{perm}}\,,$$ 
where
\begin{subequations} \label{domains}
\begin{align} 
\Omega^+(t)&=\{(x_1,x_2)\in \TT\times \RR,\;\; -1+f(x_1)<x_2<h(x_1,t)\} \,, \\
\Omega^-&=\{(x_1,x_2)\in \TT\times \RR,\;\; -2<x_2<-1+f(x_1)\}\,, \\
\Gamma_{\operatorname{perm}}&=\{(x_1,-1+f(x_1)),\;x_1\in\TT\} \,,
\end{align} 
\end{subequations}
and
where the functions $f$ and $h$ satisfy
\begin{equation}\label{notouching}
\min_{x_1 \in \TT  }f(x_1) >-1  \ \text{ and } \ h(x_1,0)>-1+f(x_1).
\end{equation}
The fixed-in-time permeability interface
$\Gamma_{\operatorname{perm}}$ denotes the curve, across  which the permeability function $\beta(x)$ is discontinuous; specifically,  
 the permeability function $\beta(x)$ is defined as
$$
\beta(x)=\left\{\begin{array}{ll}\beta^+ & \text{ in }\Omega^+(t)\\
\beta^- & \text{ in }\Omega^-\,\end{array}\right.,
$$
for given constants $\beta^\pm>0$.   The domain for this problem is also an unknown; thus, 
we must  track the evolution of the time-dependent interface
or free-boundary $\Gamma(t)$, which is defined as the set
$$
\Gamma(t)=\{(x_1,h(x_1,t)),\;x_1\in\TT\} \,.
$$

For simplicity, we shall set the fluid density $\rho$ and viscosity $\mu$  to $1$. As the fluid is incompressible, 
it follows that
$$
\jump{u\cdot n_{\operatorname{perm}}}=0\text{ on }\Gamma_{\operatorname{perm}}\times[0,T] \,.
$$

With the domains defined, the Muskat problem consists of the
following system of coupled equations:
%
\begin{subequations}\label{laplacian}
\begin{alignat}{2}
\frac{u^\pm}{\beta^\pm}+\nabla p^\pm&=-e_2,  \qquad&&\text{in } \Omega^\pm(t)\times[0,T]\,,\\
\nabla\cdot u^\pm &=0,  \qquad&&\text{in } \Omega^\pm(t)\times[0,T]\,,\\
\jump{p} &= 0 &&\text{on }\Gamma_{\operatorname{perm}}\times[0,T],\\
\jump{\nabla p\cdot n_{\operatorname{perm}}}  &= -\Bigjump{\frac{1}{\beta}}u^+\cdot n_{\operatorname{perm}} \qquad &&\text{on }\Gamma_{\operatorname{perm}}\times[0,T],\\
p^+ &= 0 &&\text{on }\Gamma(t)\times[0,T],\\
\mathcal{V}(\Gamma(t)) &= u^+\cdot n \qquad&&\text{on } \Gamma(t)\times[0,T]\,,\\
u^-\cdot e_2 &= 0 &&\text{on }\Gamma_{\operatorname{bot}}\times[0,T],
\end{alignat}
\end{subequations}
where $\mathcal{V}(\Gamma(t))$ denotes the normal component of the velocity of the time-dependent free-boundary $\Gamma(t)$, $n$ is the
 (upward) unit normal to $\Gamma(t)$, $n_{\operatorname{perm}}$ is the (upward-pointing) unit normal to $\Gamma_{\operatorname{perm}}$, and 
 $\jump{f} = f^+ - f^-$ denotes the jump of a discontinuous function $f$ across $\Gamma_{\operatorname{perm}}$. 

\subsection{A brief history of the analysis of the Muskat problem}

Darcy's law \eqref{basic-model} is a standard model for flow in aquifers, oil wells, or geothermal reservoirs, and it is therefore of  practical importance in geoscience (see, for example, \cite{CF,Parseval-Pillai-Advani:model-variation-permeability} and the references therein). Furthermore, the Muskat problem is equivalent to the Hele-Shaw cell problem with gravity
(see \cite{HeleShaw:motion-viscous-fluid-parallel-plates}) for flow between two thinly-spaced parallel plates.

There has a been a great deal of mathematical analysis of 
both the Muskat problem and the Hele-Shaw cell, and we shall only review a small fraction of the results that are, in some sense, most closely related to our result.

For the Muskat problem with a continuous permeability function, existence of solutions in the Sobolev space $H^3$ has been established by
C\'ordoba \&  Gancedo \cite{c-g07, c-g10}, C\'ordoba, C\'ordoba \&  Gancedo in \cite{c-c-g10},  and C\'ordoba, Granero-Belinch\'on \&  Orive \cite{CGO}, using the singular integral contour equation for the height function $h$.     Cheng, Granero-Belinch\'on \& Shkoller \cite{CGS} 
introduced the direct PDE approach (modified for use, herein), and established an $H^2$ existence theory (see also Cheng, Coutand \& Shkoller \cite{cheng2012global} for a similar approach to the horizontal Hele-Shaw cell problem).  This was followed by an $H^2$ existence theory by Constantin, Gancedo, Shvydkoy \& Vicol \cite{CGSVfiniteslope} using the singular integral approach; they also
obtained a finite-slope global existence result. Very recently, when surface tension effect are neglected, local existence in $H^{s}$ with $s>3/2$ has been obtained by Matioc \cite{matioc}. In the same paper, Matioc also proved an $H^s$, $s>2$ local existence result for the case where surface tension effects are considered. In the presence of surface tension, local existence in $H^6$ was also obtained by Ambrose \cite{ambrose2004well, AmbroseST}.

In the case of a discontinuous permeability function (with a jump across the flat curve $(x_1,-1)$), the local-in-time existence of solutions has been proved by 
Berselli, C\'ordoba \& Granero-Belinch\'on \cite{BCG}. In the case of two fluids with different viscosities and densities and permeability function with a jump given by an arbitrary smooth curve $(f_1(\alpha),f_2(\alpha))$ the local-in-time existence of solutions has been established by Pern\'as-Casta\~no \cite{pinhomo} .

For the case of a continuous permeability, there are a variety of results showing global existence of strong solutions under certain conditions on the initial data. In particular, Cheng, Granero-Belinch\'on \& Shkoller \cite{CGS} proved global existence under restrictions on the size of $\|h_0\|_{H^2}$, while C\'ordoba, Constantin, Gancedo \& Strain \cite{ccgs-10} and C\'ordoba, Constantin, Gancedo, Strain \& Rodr\'iguez-Piazza \cite{ccgs-13} proved global existence under restrictions on the size of $\|\widehat{h'}_0\|_{L^1}$, where $\hat{h}$ denotes the Fourier transform. The global existence of weak solution has been proved by C\'ordoba, Constantin, Gancedo \& Strain \cite{ccgs-10} and Granero-Belinch\'on \cite{G} for initial data satisfying restrictions on $\|h_0\|_{\dot{W}^{1,\infty}}$ and $\|h_0\|_{W^{1,\infty}}$, respectively. Note that the condition on $\|h_0\|_{L^\infty}$ in \cite{G} is a consequence of having a bounded porous media.

Finite time singularities of \emph{turning} type are known to occur. A turning wave is a solution which starts as a graph, but then turns-over and loses the graph property. The existence of such waves in the Rayleigh-Taylor stable regime has been established by Castro, C\'ordoba, Fefferman, Gancedo \& L\'opez-Fern\'andez \cite{ccfgl}, C\'ordoba, Granero-Belinch\'on \& Orive-Illera \cite{CGO}, Berselli, C\'ordoba \& Granero-Belinch\'on \cite{BCG} and G\'omez-Serrano \& Granero-Belinch\'on \cite{GG}. 

Finally, some deeper insight on the turning behaviour has been obtained by C{\'o}rdoba, G{\'o}mez-Serrano \& Zlato\u{s} \cite{CGSZ, CGSZ2}, where, in particular, they proved that certain solutions to the two-phase Muskat problem start as a graph, then turn-over and lose the graph property and hence violate the Rayleigh-Taylor condition but then stabilize and return to being a graph. Furthermore Castro, C\'ordoba, Fefferman \& Gancedo \cite{castro2012breakdown} also proved that there exist interfaces such that, after turning, the interface is no longer analytic and, in fact, 
$$
\limsup_{t\rightarrow T}\|z(t) \|_{C^4}=\infty,
$$
for a finite time $T>0$.

 Gancedo \& Strain \cite{gancedo2014absence} have shown that
the finite-time  \emph{splash} and \emph{splat} singularities (a self-intersection of a locally smooth interface) cannot occur
 for the two-phase Muskat problem (see also Fefferman, Ionescu \& Lie \cite{FeIoLi2016} and Coutand \& Shkoller \cite{CoSh2016}).
 However, in the case of the one-phase Muskat problem, Castro, C\'ordoba, Fefferman \& Gancedo \cite{ccfgonephase} proved that the splash singularities may occur while C\'ordoba \& Pern\'as-Casta\~no \cite{cponephase} showed that splat singularities cannot occur.   See also Coutand \& Shkoller \cite{CoSh2014} for splash and splat singularities for the 3-D Euler equations and related
models.

Let us also mention that several results for the multiphase Muskat problem have been obtained in the completely different framework of \emph{little H\"older spaces} $h^{k+\alpha}$ by Escher \& Matioc \cite{e-m10}, Escher, Matioc \& Matioc \cite{escher2011generalized} and Escher, Matioc \& Walker \cite{emw15}. 

Very recently, a regularity result in H\"{o}lder spaces for the one-phase Hele-Shaw problem has been obtained by Chang-Lara \& Guill\'en \cite{chang2016free} using the hodograph transform. Also, Pr\"uss \& Simonett \cite{pruess2016muskat} studied the two-phase Muskat problem in a more geometric framework using the Hanzawa transform. In particular, these authors show well-posedness, characterize and study the dynamic stability of the equilibria. 

Finally, using a convex integration approach, Castro, C\'ordoba \& Faraco \cite{castro2016mixing} very recently proved the existence of weak solutions for the Muskat problem in the case where the denser fluid lies above the lighter fluid, so, it is in the Rayleigh-Taylor unstable regime. Remarkably, these solutions develop a \emph{mixing zone} (a strip containing particles from both phases and, consequently, with fluid particles having both densities), growing linearly in time.

\subsection{Methodology}
As noted above, most prior existence theorems have relied upon  the  singular integral contour equation for the height function $h$; in the case 
of the infinitely deep two-phase Muskat problem with continuous permeability, the evolution equation for $h$ can be written as
\begin{eqnarray}
h_t(x_1)=\text{p.v.}\int_{\mathbb{R}}\frac{h'(x_1)-h'(x_1-y)}{y}\frac{1}{1+\left(\frac{h(x_1)-h(x_1-y)}{y}\right)^2}dy;
\label{Ieqv2}
\end{eqnarray}
see, for example,  \cite{c-g07} for the derivation.

The contour equation \eqref{Ieqv2} depends crucially on the geometry of the domain and the permeability function. 
In particular, when the porous medium has finite depth (equal to $\pi/2$) and the permeability function is discontinuous across the curve 
$(x_1,-1)$, it was shown in  \cite{BCG} that \eqref{Ieqv2} takes the form:
\begin{eqnarray}
h_t(x_1)&=&\frac{\beta^+(-\jump{\rho})}{4\pi}\text{p.v.}\int_{\mathbb{R}}\frac{(h'(x_1)-h'(y))\sinh(x_1-y)}{\cosh(x_1-y)-\cos(h(x_1)-h(y))}dy\nonumber\\
&&+\frac{\beta^+(-\jump{\rho})}{4\pi}\text{p.v.}\int_{\mathbb{R}}\frac{(h'(x_1)-h'(y))\sinh(x_1-y)}{\cosh(x_1-y)+\cos(h(x_1)+h(y))}dy\nonumber\\
&&+\frac{1}{4\pi}\text{p.v.}\int_{\mathbb{R}}\frac{\varpi_2(y)(\sinh(x_1-y)+h'(x_1)\sin(h(x_1)+1))}{\cosh(x_1-y)-\cos(h(x_1)+1)}dy\nonumber\\
&&+\frac{1}{4\pi}\text{p.v.}\int_{\mathbb{R}}\frac{\varpi_2(y)(-\sinh(x_1-y)+h'(x_1)\sin(h(x_1)-1))}{\cosh(x_1-y)+\cos(h(x_1)-1)}d\beta\label{IIeqv2},
\end{eqnarray}
where
\begin{eqnarray}
\varpi_2(x_1)&=&\frac{\beta^+-\beta^-}{\beta^++\beta^-}\frac{\beta^+(-\jump{\rho})}{2\pi}\text{p.v.}\int_{\mathbb{R}} h'(y)\frac{\sin(1+h(y))dy}{\cosh(x_1-y)-\cos(1+h(y))}\nonumber\\
&&-\frac{\beta^+-\beta^-}{\beta^++\beta^-}\frac{\beta^+(-\jump{\rho})}{2\pi}\text{p.v.}\int_{\mathbb{R}} h'(y)\frac{\sin(-1+h(y))dy}{\cosh(x_1-y)+\cos(-1+h(y))}\nonumber\\
&&+\frac{\left(\frac{\beta^+-\beta^-}{\beta^++\beta^-}\right)^2}{\sqrt{2\pi}}\frac{\beta^+(-\jump{\rho})}{2\pi}G_{\beta}*\text{p.v.}\int_{\mathbb{R}}\frac{h'(y)\sin(1+h(y))dy}{\cosh(x_1-y)-\cos(1+h(y))}\nonumber\\
&&-\frac{\left(\frac{\beta^+-\beta^-}{\beta^++\beta^-}\right)^2}{\sqrt{2\pi}}\frac{\beta^+(-\jump{\rho})}{2\pi}G_{\beta}*\text{p.v.}\int_{\mathbb{R}}\frac{h'(y)\sin(-1+h(y))dy}{\cosh(x_1-y)+\cos(-1+h(y))}, \label{IIw2defc}
\end{eqnarray}
with 
$$
G_{\beta}(x_1)=\mathcal{F}^{-1}\left(\frac{\mathcal{F}\left(\frac{\sin(2)}{\cosh(x_1)+\cos(2)}\right)(\zeta)}{1+\frac{\frac{\beta^+-\beta^-}{\beta^++\beta^-}}{\sqrt{2\pi}}\mathcal{F}\left(\frac{\sin(2)}{\cosh(x_1)+\cos(2)}\right)(\zeta)}\right) \,,
$$
a Schwartz function, and where $\mathcal{F}$ denotes the Fourier transform. 
Let us emphasize that, due to the non-local character of $\varpi_2$ given by \eqref{IIw2defc}, the contour equation \eqref{IIeqv2} is
significantly more challenging to analyse than \eqref{Ieqv2}.
Note also, from the definition of $G_{\beta}(x_1)$ that the highly non-local convolution terms in \eqref{IIw2defc} are not explicitly defined.

Because of the complications inherent in the singular integral approach of \eqref{IIw2defc}, we shall instead analyze the system \eqref{laplacian}
directly.  As \eqref{laplacian} is set on the time-dependent {\it a priori} unknown domain $\Omega(t)$, in order to build an existence theory,
we first pull-back this system of equations onto a fixed-in-time spatial domain.
We use a  carefully chosen change-of-variables that transforms the free-boundary problem \eqref{laplacian} into a system of
equations set on a smooth and fixed domain, but having 
 time-dependent coefficients.
 
 To pull-back our problem, 
we employ a  family of
 diffeomorphisms $\psi^\pm$ which are elliptic extensions of the interface parametrizations, and thus have optimal $H^s$ Sobolev regularity. 
 The time-dependent coefficients (in the pulled-back description) arise from differentiation and inversion of these maps $\psi^\pm$;  by studying the transformed 
 Darcy's Law, we obtain a new higher-order energy integral that provides the regularity of the moving interface $\Gamma(t)$.   Additionally,
 we obtain an $L^2$-in-time  parabolic regularity gain, analogous to the regularity gain for solutions to the heat equation, except that we
 gain a $1/2$-derivative in space rather than a full derivative.  The regularity of the interface $\Gamma(t)$ as well as the improved
 $L^2$-in-time parabolic regularity gain are found from the non-linear structure of the pulled-back representation of the Muskat problem.  In
 particular, we do no rely on the explicit structure of the singular integral contour equation, and as such, we are free to study general domain
 geometries and permeability functions.
 

\subsection{The main results}
As we will show, the Rayleigh-Taylor (RT) stability condition, given by $ - \frac{\partial p}{\partial n} > 0$ on $\Gamma(t)$, 
is a sufficient condition for well-posedness of the Muskat problem \eqref{laplacian} in Sobolev spaces.   
 In particular, with
$$
 p_0 :=p ( \cdot , 0) \ \text{ and } \
\Gamma := \Gamma(0)\,, $$
and letting
$N:= n( \cdot , 0)$ denote the outward unit normal to $\Gamma$, we prove that
for any initial interface $\Gamma$ of arbitrary size and of class $H^2$,
 chosen such that the RT stability condition
\begin{equation}\label{RT}
- \frac{\partial p_0}{\partial N} > 0 \ \text{ on } \ \Gamma
\end{equation}
is satisfied, 
there exists a  unique  solution $(u^\pm(x,t),p^\pm(x,t),h(x_1,t))$ to the one-phase Muskat problem with discontinuous permeability function.

More precisely, we prove the following
\begin{theorem}[Local well-posedness in $H^2$]\label{localsmall} Suppose the initial interface $\Gamma$ is given as the graph
$(x_1, h_0(x_1))$ where  $h_0\in H^2(\TT)$ and $\int_{\TT} h_0(x_1) dx_1 =0$, and  that the RT condition \eqref{RT} is satisfied.
Let $\Gamma_{\operatorname{perm}}$ be given as the graph 
$(x_1, -1+f(x_1))$ for a function $f\in H^{2.5}(\TT)$. Assume also that \eqref{notouching} holds. Then, there exists a time $T(h_0,f)>0$ and a unique solution  
\begin{align*} 
&h\in C([0,T(h_0,f)];H^2(\TT))\cap L^2(0,T(h_0,f);H^{2.5}(\TT)) \,, \\
&u^\pm\in C([0,T(h_0,f)];H^{1.5}(\Omega^\pm(t)))\cap L^2(0,T(h_0,f);H^{2}(\Omega^\pm(t))) \,, \\
&
p^\pm\in C([0,T(h_0,f)];H^{2.5}(\Omega^\pm(t)))\cap L^2(0,T(h_0,f);H^{3}(\Omega^\pm(t))) \,,
\end{align*} 
to the system \eqref{laplacian},
satisfying
$$
\|h(t)\|_{L^2(\TT)}^2+2\int_0^t\left\|\frac{u^+(s)}{\beta^+}\right\|_{L^2(\Omega^+(s))}^2d s
+2\int_0^t\left\|\frac{u^-(s)}{\beta^-}\right\|_{L^2(\Omega^-)}^2ds  =\|h_0\|_{L^2(\TT)}^2,
$$
and
\begin{align*} 
&
\|h\|_{C([0,T(h_0,f)],H^2(\TT))}+\|h_t\|_{L^2(0,T(h_0,f);H^{1.5}(\TT))}+\|h\|_{L^2(0,T(h_0,f);H^{2.5}(\TT))} \\
& \qquad + 
\|p\|_{C([0,T(h_0,f)],H^{2.5}(\Omega^+(t)\cup\Omega^-))}+\|p\|_{L^2(0,T(h_0,f);H^{3}(\Omega^+(t)\cup\Omega^-))} \\
& \qquad +
\|u\|_{C([0,T(h_0,f)],H^{1.5}(\Omega^+(t)\cup\Omega^-))}+\|u\|_{L^2(0,T(h_0,f);H^{2}(\Omega^+(t)\cup\Omega^-))}\leq C(h_0,f) 
\end{align*} 
for a constant $C(h_0,f)$ which depends on $h_0$ and $f$.
\end{theorem}

\begin{remark}
It is easy to see that if $h_0(x_1)=f(x_1)=0$, the solution is given by
\begin{equation}\label{gs001}
u^\pm(x,t)=0,\;h(x_1,t)=0,\;p^\pm(x,t)=-x_2,
\end{equation} 
and the RT condition is satisfied. There exist infinitely many initial data $h_0$ satisfying the RT condition;  for example, small
perturbations  of \eqref{gs001} satisfy the RT condition \eqref{RT} via implicit function theorem arguments.
\end{remark}

\begin{theorem}[Global well-posedness and decay to equilibrium in $H^2$]\label{globalsmall} Suppose the initial interface $\Gamma$ is given as the graph
$(x_1, h_0(x_1))$ where  $h_0\in H^2(\TT)$ and $\int_{\TT} h_0(x_1) dx_1 =0$. Let $\Gamma_{\operatorname{perm}}$ be given as the graph 
$(x_1, -1)$. Then, there exists a constant $\mathscr{C}$ such that if
$$
|h_0|_2<\mathscr{C},
$$
the RT condition \eqref{RT} is satisfied and there exists a unique solution 
\begin{align*} 
&h\in C([0,\infty);H^2(\TT))\cap L^2(0,\infty;H^{2.5}(\TT)) \,, \\
&u^\pm\in C([0,\infty);H^{1.5}(\Omega^\pm(t)))\cap L^2(0,\infty;H^{2}(\Omega^\pm(t))) \,, \\
&
p^\pm\in C([0,\infty);H^{2.5}(\Omega^\pm(t)))\cap L^2(0,\infty;H^{3}(\Omega^\pm(t))) \,,
\end{align*} 
to the system \eqref{laplacian},
satisfying
\begin{align*} 
\|h(t)\|_{H^2(\TT)}\leq \|h_0\|_{H^2(\TT)}e^{-\gamma t/2} 
\end{align*} 
for a constant $\gamma(h_0,\beta^\pm)$, which depends on $h_0$ and $\beta^\pm$.
\end{theorem}

\begin{remark}
Note that the question of whether the free boundary $\Gamma(t)$ can reach the curve $\Gamma_{\operatorname{perm}}$ in finite time, in a a situation that resembles the splash/splat singularity, remains an open problem. In fact, such behavior can be seen as a singular phenomena (for instance, some of the (non-singular) terms in \eqref{IIeqv2} and \eqref{IIw2defc} become singular integral operators). As Theorem \ref{globalsmall} implies that $\Gamma(t)$ cannot reach the curve $\Gamma_{\operatorname{perm}}$ in finite time if $h_0$ is small enough, this result rules out the possibility of interface collision in finite time for small initial data. 
\end{remark}

\begin{remark} We note that the \emph{dry zone} (the region without fluid) lies above the curve $\Gamma(t)$, and so, as long as \eqref{notouching} holds, the dry zone lies above $\Gamma_{\operatorname{perm}}$. The question of whether a \emph{dry zone} can form inside $\Omega^-$ remains an open problem. In other words, assume that there exists a solution $h(x_1,t)$ up to time $T$ and assume also that $\Gamma(t)$ intersects $\Gamma_{\operatorname{perm}}$ at the point $(x_0,t')\in \TT\times(0,T)$, \emph{i.e.}
$$
h(x_0,t')=-1+f(x_0).
$$
Then, it is not clear if the curve $\Gamma(t)$ may \emph{cross} the curve $\Gamma_{\operatorname{perm}}$, \emph{i.e.}
$$
h(x_1,t)<-1+f(x_1),\,\forall\, (x_1,t)\in(x_0-\epsilon,x_0+\epsilon)\times (t',t'+\delta),\,,
$$
for certain $\epsilon,\delta>0$. Note also that, if this happens, then the region
$$
\{(x_1,x_2), x_1\in(x_0-\epsilon,x_0+\epsilon),\;h(x_1,t)<x_2<-1+f(x_1)\}\subset \Omega^-
$$
is contained in the dry zone. 
\end{remark}

\begin{remark}
The exponential decay of the solution $h(t)$ is a consequence of an \emph{energy-energy dissipation inequality} establishing a relationship between the interface regularity and the regularity of the semi-ALE velocity (see Sections \ref{semiale} and \ref{secglobal}):
$$
\|h''(t)\|_{L^2(\TT)}\leq C\|w''\|_{L^2(\TT\times (-2,-1)\cup\TT\times (-1,0))}.
$$
\end{remark}
\begin{remark}
Note that the linearized evolution equation for a small perturbation of the flat interface can be written as
$$
h_t=-\Lambda_{\refdom^+} h,
$$
where $\Lambda_{\refdom^+}$ is the Dirichlet-to-Neumann map associated with the elliptic system \eqref{deltapsi+}:
$$
\Lambda_{\refdom^+}h(x_1)=\delta\psi,_2(x_1,0).
$$
An integration by parts shows that
\begin{align*}
\int_{\mathbb{S}}\int_{-1}^0\delta\psi^+\Delta\delta\psi^+dx_2dx_1&=\int_{\mathbb{S}}\delta\psi^+(x_1,0)\delta\psi^+,_2(x_1,0)dx_1
-\int_{\mathbb{S}}\delta\psi^+(x_1,-1)\delta\psi^+,_2(x_1,-1)dx_1\\
&\quad-\int_{\mathbb{S}}\int_{-1}^0|\nabla\delta\psi^+|^2dx_2dx_1,
\end{align*}
so that, 
$$
\int_\Gamma \Lambda_{\refdom^+} h hdx_1=\int_{\refdom^+}|\nabla\delta\psi^+|^2dx.
$$
We also have the following Poincar\'e-Wirtinger inequality
\begin{align*}
\int_{\refdom^+}|\delta\psi^+(x)|^2dx&= \int_{\refdom^+}\left|\int_0^1\delta\psi^+,_2(x_1,sx_2+(1-s)(-1))(x_2+1)ds\right|^2dx\\
&\leq\int_{\mathbb{S}}\int_{-1}^0\int_0^1\left|\delta\psi^+,_2(x_1,sx_2+(1-s)(-1))\right|^2(x_2+1)^2dsdx_2dx_1\\
&=\int_{-1}^0(x_2+1)\int_{\mathbb{S}}\int_{-1}^{x_2}\left|\delta\psi^+,_2(y)\right|^2dy_2dy_1dx_2\\
&\leq \int_{-1}^0(x_2+1)\int_{\refdom^+}\left|\delta\psi^+,_2(y)\right|^2dydx_2\\
&\leq \frac{1}{2}\int_{\refdom^+}\left|\delta\psi^+,_2(y)\right|^2dy.
\end{align*}
Thus, using the trace theorem, we conclude that
$$
\int_\Gamma \Lambda_{\refdom^+} h hdx_1\geq 0.5\|\delta\psi^+\|_{1,+}^2\geq \nu|h|_{0.5}^2\geq \nu|h|_0^2,
$$
for $\nu>0$. Exponential decay for the nonlinear problem (under smallness assumptions) is hence also expected.
\end{remark}
\subsection{Notation used throughout the paper}
For a matrix  $A$, we write $A^i_j$ for the component of $A$ located in row $i$ and column $j$. 
We use the Einstein summation convention, wherein repeated indices are summed from $1$ to $2$.  We denote the $j$th canonical
basis vector in $\RR^2$ by $e_j$.

For $s\ge 0$, we set
$$
\|u\|_{s,+} := \|u^+\|_{H^s(\refdom^+)}\,,\ \|u\|_{s,-} := \|u^-\|_{H^s(\refdom^-)}\,,\ \|u\|_{s,\pm} := \|u^+\|_{s,+} + \|u^-\|_{s,-}
$$
and
$$
|h|_s := \|h\|_{H^s(\Gamma)} \,.
$$
For functions $h$ defined on $\Gamma_{\operatorname{perm}}$, we shall also denote the $H^s$ norm by 
 $|h|_s := \|h\|_{H^s(\Gamma_{\operatorname{perm}})}$, whenever the context is clear.

We write
$$
f' = \frac{\p f}{\p x_1} \,, \ 
f,_{k}=\frac{\partial f}{\partial x_k} \,, \text{ and } f_t = \frac{\partial f }{\partial t} \,.
$$
For a diffeomorphism $\psi$, we let $A=(\nabla\psi)^{-1}$, and define
\begin{align} 
\curl_{\psi} v& =A^j_1v^2,_j-A^j_2v^1,_j \,, \label{divpsi}\\
\div_{\psi} v&=A^i_jv^j,_{i} \label{curlpsi}\,.
\end{align}

\section{The Muskat problem in the ALE formulation}\label{sec2}
\subsection{Constructing the family of diffeomorphisms $\psi( \cdot ,t)$}\label{subsection_psi}
\subsubsection{The idea for the construction}

Our analysis of the Muskat problem \eqref{laplacian} is founded on a time-dependent change-of-variables which converts the free boundary problem to one set
on smooth reference domains $\refdom^\pm$
\begin{equation}\label{ref_domain}
\refdom^+=\TT\times(-1,0)\,,\refdom^-=\TT\times(-2,-1)\,, \quad 
\end{equation}
The boundaries of the domains $\refdom^\pm$ are defined as 
\begin{equation}\label{ref_domain2}
\Gamma_{\operatorname{bot}}=\{(x_1,-2),x_1\in\TT\}\,, \Gamma_{\operatorname{perm}}=\{(x_1,-1),x_1\in\TT\}\,, \text{ and }\Gamma=\{(x_1,0),x_1\in\TT\}\,.
\end{equation} 
We let $N=e_2$ denote
the unit normal vector on $\Gamma$ (outwards), $\Gamma_{\operatorname{perm}}$ and $\Gamma_{\operatorname{bot}}$. 

As our analysis crucially relies on obtaining a parabolic regularity gain, we need a reference domain $\refdom^+$ with  $C^\infty$ boundary. In particular, the initial domain $\Omega^+(0)$ cannot serve as a reference domain.

We adapt the ideas from \cite{CGS} to construct the time-dependent family of diffeomorphisms with optimal Sobolev regularity, $\psi(x,t)$, that we shall use to pull-back \eqref{laplacian} onto
the fixed domain $\refdom^\pm$. Before detailing this construction, let us sketch the procedure. First, we construct a diffeomorphism with optimal Sobolev regularity at $t=0$:
$$
\psi^+(0):\refdom^+\to\Omega^+(0),\;\psi^-:\refdom^-\to\Omega^-.
$$
To do so we follow a three step procedure:

\begin{itemize}
\item For $0<\delta\ll1$ a sufficiently small parameter (to be fixed later), we define auxiliary domains, $\refdom^{\pm,\delta}(0)$. These auxiliary domains are constructed via mollification of $h(x_1,0)$ and $f(x_1)$ and, thus, they are infinitely smooth. We define the graph diffeomorphism 
$$
\phi_1^\pm:\refdom^{\pm}\rightarrow \refdom^{\pm,\delta}.
$$ 
These diffeomorphisms are of class $C^\infty$ because of the smoothness of the domains $\refdom^{\pm}$ and $\refdom^{\pm,\delta}$.
\item We need another diffeomorphisms from the auxiliary domain $ \refdom^{\pm,\delta}$ to $ \Omega^{\pm}(0)$. We need these diffeomorphisms to gain 1/2 derivatives with respect to the regularity of $\Omega^\pm(0)$. In order that this optimal regularity is obtained, we make use of the definition of $\refdom^{\pm,\delta}$ and the properties of our mollifiers. We define 
$$
\phi_2^\pm:\refdom^{\pm,\delta}\rightarrow \Omega^\pm(0)
$$ as the solution to Laplace problems with appropriate boundary conditions. Using the boundary data and the inverse function theorem, these mappings $\phi_2^\pm$ are $H^{2.5}-$class diffeomorphisms. 

\item Finally, we define 
$$
\psi^+(0)=\phi_2^+ \circ \phi_1^+,\;\psi^-=\phi_2^- \circ \phi_1^-.
$$
As composition of diffeomorphisms, $\psi^\pm(0)$ is a diffeomorphism.
\end{itemize}

Once the initial diffeomorphism with optimal Sobolev regularity is constructed, we solve  Poisson problems (to be detailed below) for $\psi^\pm(x,t)$. An  application of the inverse function theorem together with standard elliptic estimates will show that these mappings $\psi^\pm(x,t)$ are a family of diffeomorphisms with the desired smoothness.


\subsubsection{Constructing the initial \emph{regularizing} diffeomorphism $\psi( \cdot ,0)$}
Given a function $h\in C(0,T;H^2)$ with initial data $h(x_1, 0)=h_0(x_1)$, we fix $0<\delta\ll1$ and define our auxiliary domains and boundaries
\begin{equation*}
\refdom^{+,\delta}(0)=\{(x_1,x_2),\, x_1\in\TT,\,-1+\jdel f(x_1)<x_2<\jdel h_0(x_1)\},
\end{equation*}
\begin{equation*}
\refdom^{-,\delta}=\{(x_1,x_2),\, x_1\in\TT,\,-2<x_2<-1+\jdel f(x_1)\},
\end{equation*}
\begin{equation*}
\Gamma^\delta(0)=\{(x_1,\jdel h_0(x_1)),\, x_1\in\TT\},\,\Gamma_{\operatorname{perm}}^\delta=\{(x_1,-1+\jdel f(x_1)),\, x_1\in\TT\}.
\end{equation*}
As we said previously, we define the graph diffeomorphism 
\begin{equation*}
\phi^+_1(x_1,x_2)=\left(x_1,(x_2+1)\jdel h_0(x_1)-(-1+\jdel f(x_1))x_2\right),
\end{equation*}
\begin{equation*}
\phi^-_1(x_1,x_2)=\left(x_1,x_2+\jdel f(x_1)(x_2+2)\right),
\end{equation*}
where $\jdel$ denotes the convolution with a standard Friedrich's mollifier. This function 
$$
\phi_1^\pm:\refdom^{\pm}\rightarrow\refdom^{\pm,\delta}(0)
$$ 
is a $C^\infty$ diffeomorphism.

Next, we have to define the regularizing diffeomorphisms 
$$
\phi^\pm_2:\refdom^{\pm,\delta}(0)\rightarrow\Omega^\pm(0).
$$ 
We define these mappings as the solution to the following elliptic problems:
\begin{subequations}\label{phi2}
\begin{alignat}{2}
\Delta \phi^+_2&= 0  \qquad&&\text{in}\quad \refdom^{+,\delta}(0)\,,\\
\phi^+_2 &= (x_1,x_2)+[h_0(x_1)-\jdel h_0(x_1)]e_2 \qquad &&\text{on}\quad \Gamma^\delta(0)\,,\\
\phi^+_2 &= (x_1,x_2)+[f(x_1)-\jdel f(x_1)]e_2 &&\text{on}\quad \Gamma^\delta_{\operatorname{perm}} \,,
\end{alignat}
\end{subequations}
\begin{subequations}\label{phi22}
\begin{alignat}{2}
\Delta \phi^-_2&= 0  \qquad &&\text{in}\quad \refdom^{-,\delta}\,,\\
\phi^-_2 &= (x_1,x_2)+[f(x_1)-\jdel f(x_1)]e_2 \qquad &&\text{on}\quad \Gamma^\delta_{\operatorname{perm}} \,\\
\phi^-_2 &= (x_1,x_2) \qquad &&\text{on}\quad \Gamma_{\operatorname{bot}} \,.
\end{alignat}
\end{subequations}
Using standard elliptic regularity theory, we have that
$$
\|\phi_2 - e\|_{H^{2.5}(\refdom^{\pm,\delta})}\leq C(|h_0-\jdel h_0|_2+|f-\jdel f|_2),
$$
where $e=(x_1,x_2)$ denotes the identity mapping. Using the Sobolev embedding theorem,  and taking $\delta >0$ sufficiently small, we have that
$$
\|\phi_2 - e\|_{C^{1}(\refdom^{\pm,\delta})}\ll 1,
$$
so, due to the inverse function theorem, we obtain that $\phi^\pm_2$ is an $H^{2.5}$-class diffeomorphism. 

As in \cite{CGS}, we define
\begin{equation}\label{psi0a}
\psi^+(0)=\phi^+_2\circ\phi^+_1:\refdom^+\rightarrow \Omega^+(0),\;\;\psi^-=\phi^-_2\circ\phi^-_1:\refdom^-\rightarrow \Omega^-.
\end{equation}
Then, this mapping is also an $H^{2.5}$-class diffeomorphism. 


\subsubsection{Constructing the time-dependent family of \emph{regularizing} diffeomorphisms $\psi( \cdot ,t)$}
We define the time-dependent family of diffeomorphisms $\psi(t) = \psi( \cdot, t)$ as solutions to Poisson equations with forcing depending on $\psi(0)$. The main point of this construction is that due to the continuity in time of the interface $h$ and standard elliptic estimates, the time-dependent family of diffeomorphisms $\psi(t) = \psi( \cdot, t)$ is going to remain close to the initial diffeomorphism $\psi(0)$.

In particular, we consider the following elliptic system:
\begin{subequations}\label{psita}
\begin{alignat}{2}
\Delta \psi^+(t)&= \Delta \psi^+(0)  \qquad&&\text{in}\quad \refdom^+\times[0,T]\,,\\
\psi^+(t)&= (x_1,x_2)+h(x_1,t)e_2 \qquad &&\text{on}\quad \Gamma\times[0,T]\,,\\
\psi^+(t)&= (x_1,x_2)+f(x_1)e_2 \qquad &&\text{on}\quad \Gamma_{\operatorname{perm}}\times[0,T]\,.
\end{alignat}
\end{subequations}
and $\psi^-(t)= \psi^-$. Because of the forcing term present in (\ref{psita}a), we have that $\psi^+(t)-\psi^+(0)$ solves
\begin{subequations}
\begin{alignat*}{2}
\Delta (\psi^+(t)-\psi^+(0))&=0  \qquad&&\text{in}\quad \refdom^+\times[0,T]\,,\\
\psi^+(t)-\psi^+(0)&= (h(x_1,t)-h(x_1,0))e_2 \qquad &&\text{on}\quad \Gamma\times[0,T]\,,\\
\psi^+(t)-\psi^+(0)&= 0 \qquad &&\text{on}\quad \Gamma_{\operatorname{perm}}\times[0,T]\,.
\end{alignat*}
\end{subequations}

Due to elliptic estimates, we have the bound
\begin{align}\label{closetopsi0}
\|\psi(t)-\psi(0)\|_{2.25,\pm}&\leq C|h(t)- h_0|_{1.75}.  
\end{align}
By taking sufficiently small time $t$ and recalling that $h\in C(0,T;H^2)$, we have that
$$
\|\psi(t)\|_{2.25,\pm}\leq C|h(t)- h_0|_{1.75}+C(|h_0|_{1.75}+|f|_{1.75}+1)\leq 2C(|h_0|_{1.75}+|f|_{1.75}+1).
$$
Writing 
$$
J(t)=\text{det}(\nabla\psi(t))=\psi^1,_1\psi^2,_2-\psi^2,_1\psi^1,_2,
$$
we have the bound
\begin{align}\label{J1.25}
\|J(t)-J(0)\|_{1.25,\pm}&\leq C|h(t)- h_0|_{1.75}. 
\end{align}
Consequently, using $h\in C(0,T;H^2)$, for sufficiently small time $t$, we have that
$$
\min_{x\in\refdom^\pm}\frac{J(0)}{2}< J(t)< 2\max_{x\in\refdom^\pm} J(0),
$$
and, thanks to \eqref{closetopsi0}, we have that
$$
\|\psi(t)-\psi(0)\|_{C^1}\leq C|h(t)- h_0|_{1.75}\ll 1.  
$$
Due to the inverse function theorem and using the fact that $\psi(0)$ is a diffeomorphism, we see that 
$$
\psi^\pm(t):\refdom^\pm \to \Omega^\pm(t)
$$ 
is a diffeomorphism.  From the elliptic estimate
$$
\|\psi(t)\|_{2.5,\pm}\leq  C(|h(t)|_{2}+|f|_{2}+1),
$$ 
we have that $\psi(t)$ is an $H^{2.5}$-class diffeomorphism. 
\subsubsection{The matrix $A( \cdot ,t)$}
We write $A=(\nabla\psi)^{-1}$. Thus, 
$$
A^i_r\psi^r,_{j}=\delta^i_j
$$
and we obtain the useful identities
\begin{equation}\label{propA}
(A_t)^i_k=-A^i_r(\psi_t)^r,_j A^j_k,\qquad A''=-2 A'\nabla\psi' A-A\nabla\psi'' A.
\end{equation}
We will also make use of the  Piola's identity:  $(JA^k_i),_k=0$.

\subsection{The Muskat problem in the reference domains $\refdom^\pm$}
With $\psi(t)= \psi( \cdot ,t)$ defined in Section \ref{subsection_psi}, we define our new variables in the reference domains $\refdom^\pm$: $v=u\circ\psi,q=p\circ\psi$. 

We let 
$$
\tilde{\tau}=\psi',\,\tilde{n}=(\psi')^\perp,\,g= |\psi'|^2
$$
denote the (non-normalized) tangent and normal vectors and the induced metric, respectively, on $\Gamma(t)$. We also define the unit tangent vector $\tau=\tilde{\tau}/\sqrt{g}$ and the unit normal vector $n=\tilde{n}/\sqrt{g}$. In the same way, we define $\tilde{\tau}_{\operatorname{perm}},\,\tilde{n}_{\operatorname{perm}},\,g_{\operatorname{perm}},\tau_{\operatorname{perm}},\,n_{\operatorname{perm}}$ as the analogous quantities on $\Gamma_{\operatorname{perm}}$. Recall that 
$$
JA^k_iN^k=\tilde{n}^i\text{ on } \Gamma,\;JA^k_iN^k=\tilde{n}^i_{\operatorname{perm}}\text{ on } \Gamma_{\operatorname{perm}}.
$$

Hence, the ALE representation of the one-phase inhomogeneous Muskat problem is given by
\begin{subequations}\label{HS_ALE}
\begin{alignat}{2}
\frac{(v^\pm)^i}{\beta^\pm}+(A^\pm)^k_i(q^\pm+\psi^\pm\cdot e_2),_{k}&=0 \qquad&&\text{in}\quad\refdom^\pm\times[0,T]\,,\\
(A^\pm)^k_i(v^\pm)^i,_k&=0 \qquad&&\text{in}\quad\refdom^\pm\times[0,T]\,,\\
h_t(t)&= (v^+)^i J^+ (A^+)^j_i N^j\qquad&&\text{on}\quad\Gamma\times[0,T]\,,\\
q^+&=0 &&\text{on}\quad\Gamma\times[0,T]\,,\\
\jump{q}&=0 &&\text{on}\quad\Gamma_{\operatorname{perm}}\times[0,T]\,,\\
\jump{q,_k A^k_i J A^j_i N^j}&=-\Bigjump{\frac{1}{\beta}}v^i J A^j_i N^j &&\text{on}\quad\Gamma_{\operatorname{perm}}\times[0,T]\,\\
v^-_2&=0 &&\text{on}\quad\Gamma_{\operatorname{bot}}\times[0,T]\,.
\end{alignat}
\end{subequations}

\section{\emph{A priori} estimates}\label{sec3}
In this section we establish the \emph{a priori} estimates for the one-phase Muskat problem with discontinuous permeability \eqref{laplacian}.

We define the higher-order energy function 
$$
E(t)=\max_{0\leq s\leq t}|h(s)|^2_{2}+\int_0^t\|v(s)\|_{2,\pm}^2+|h(s)|_{2.5}^2ds.
$$

\begin{remark}
Another possible definition for a higher-order energy function is (see \cite{CGS})
$$
\mathcal{E}(t)=\max_{0\leq s\leq t}|h(s)|^2_{2}+\int_0^t\|v(s)\|_{2,\pm}^2.
$$
In fact, as will be shown,
$$
E(t)\leq C\mathcal{E}(t).
$$
\end{remark}

As in \cite{CGS}, our goal is to obtain the polynomial inequality 
$$
E(t)\leq \mathcal{M}_0+\mathcal{Q}(E(t))t^{\alpha},
$$
for certain $\alpha>0$, a generic polynomial $\mathcal{Q}$, and a constant $\mathcal{M}_0$ depending on $h_0$ and $f$. When $E(t)$ is continuous, the previous inequality implies the existence of $T^*(h_0,f)$ such that 
\begin{equation}\label{conclusion}
E(t)\leq 2\mathcal{M}_0.
\end{equation}

We assume that we have a smooth solution defined for $t\in[0,T]$. We take $0<T\leq 1$ small enough such that the following conditions hold: for a fixed constant $0<\epsilon\ll1$ (possibly depending on $h_0$ and $f$) and for $t\in[0,T]$,
\begin{subequations}\label{bootstrap}
\begin{align}
\|\psi(t)-\psi(0)\|_{L^\infty}+\|A(t)-A(0)\|_{L^\infty}+\|J(t)-J(0)\|_{L^\infty}&\leq \epsilon \,;\\
\|h(t)-h_0\|_{L^\infty}+\|\nabla q(t)-\nabla q(0)\|_{L^\infty}&\leq \epsilon\,;\\
E(t)&\leq 3\mathcal{M}_0\,;\\
\min_{0\leq t \leq T}\min_{x_1\in\TT}q,_2(t)&\geq \min_{x_1\in\TT}q,_2(0)/4.
\end{align}
\end{subequations}

We will show that conditions even stricter than (\ref{bootstrap}c,d) actually holds. Let us emphasize that, due to the RT condition, we have that
$$
\min_{x_1\in\TT}q,_2(0)>0.
$$

Again, we let $C=C(h_0,f,\delta)$ denote a constant that may change from line to line. We let $\mathcal{P}(x)$ denote a polynomial with coefficients that may depend on $h_0(\cdot):=h(\cdot,0), f, \delta$. This polynomial may change from line to line. 

\subsection{Estimates for some lower-order norms}

In the following, we collect some estimates of lower-order norms. The proofs are similar to those in \cite{CGS}, so, we omit them.

\begin{lemma}[Estimates for some lower-order norms of $h$, \cite{CGS}, Section 8.4.1]\label{lower2}
Given a smooth solution to the Muskat problem (\ref{HS_ALE}a-g),
\begin{subequations}\label{ht1}
\begin{align}
\int_0^t |h_{t}|^2_{1}ds&\leq  C \, E(t).\\
|h(t)-h_0|_{1}&\leq C \sqrt{E(t)}t^{1/2}.
\end{align}
\end{subequations}
\end{lemma}

\begin{lemma}[Estimates for some lower-order norms of the ALE mapping $\psi$, \cite{CGS}, Section 8.4.2]\label{lower}
Given a smooth solution to the Muskat problem (\ref{HS_ALE}a-g), 
\begin{subequations}\label{boundpsi}
\begin{align}
\|\psi(t)\|_{2.5,\pm}\leq  C(1+|h(t)|_{2}),\;\;\|\psi(t)\|_{3,\pm}&\leq  C(1+|h(t)|_{2.5})\\
\|\psi(t)-\psi(0)\|_{2.25,\pm}+\|A(t)-A(0)\|_{1.25,\pm} +\|J(t)-J(0)\|_{1.25,\pm} &\leq\sqrt[4]{t}C\sqrt{E(t)}.
\end{align}
\end{subequations}
\end{lemma}
Notice that (\ref{boundpsi}b) implies a stricter version of (\ref{bootstrap}a).
As a consequence we obtain that $|h(t)|_{1.75}, \|\psi(t)\|_{2.25,\pm}, \|J(t)\|_{1.25,\pm}$ and $\|A(t)\|_{1.25,\pm}$ are bounded by $C(h_0,f)$ uniformly for all $t\in[0,T]$. Furthermore, we also have that
\begin{equation}\label{linearJ2}
0<\frac{1}{2}\min_{x\in \refdom^+\cup\refdom^-} J(0)\leq J(t)\leq 1.5\max_{x\in \refdom^+\cup\refdom^-} J(0). 
\end{equation}

\subsection{Basic $L^2$ energy law}
\begin{lemma}[Estimates for some lower-order norms of $v$]
For a smooth solution to the Muskat problem (\ref{HS_ALE}a-g), \begin{equation}\label{lowv}
|h (t)|^2_0+2\int_0^t\left\|\sqrt{\frac{J}{\beta}}v\right\|_{0,\pm}^2ds=|h_0|_0^2\,.
\end{equation}
\end{lemma}
\begin{proof}
We test the equation (\ref{HS_ALE}a) against $Jv$ and integrate. Using Piola's identity, integrating by parts and using the divergence free condition (\ref{HS_ALE}b), we obtain that
\begin{multline*}
\left\|\sqrt{\frac{J}{\beta}}v\right\|_{0,\pm}^2+\int_{\Gamma}v^i JA^k_i(q+\psi\cdot e_2)N^kdx_1\\
-\int_{\Gamma_{\operatorname{perm}}}\jump{v^i JA^k_i(q+\psi\cdot e_2)N^k}dx_1-\int_{\Gamma_{\operatorname{bot}}}v^i JA^k_i(q+\psi\cdot e_2)N^kdx_1=0.
\end{multline*}
Then, using the jump and boundary conditions on $\Gamma_{\operatorname{perm}}$ and $\Gamma_{\operatorname{bot}}$, 
$$
\left\|\sqrt{\frac{J}{\beta}}v\right\|_{0,\pm}^2+\frac{1}{2}\frac{d}{dt}|h|_0^2=0.
$$
\end{proof}

\subsection{Estimates for $h\in L^2(0,T;H^{2.5}(\Gamma))$ and $h_t\in L^2(0,T;H^{1.5}(\Gamma))$}
From (\ref{HS_ALE}a)
\begin{equation}\label{h2.5pointwise}
(v_i+\beta^+\delta^2_i)\tau_i=0\text{ on }\Gamma,\,\text{ and } v_i'\tau_i=-\beta^+JA^2_iA^2_iq,_2\frac{h''}{g^{3/2}} \text{ on }\Gamma \,.
\end{equation}

\begin{lemma}[Parabolic smoothing, \cite{CGS}, Section 8.4.6]\label{parabolic}
Given a smooth solution to the Muskat problem (\ref{HS_ALE}a-g), 
$$
h\in C([0,T],H^2(\TT)).
$$
In particular, 
\begin{equation}\label{ht1.5}
\int_0^t |h_{t}(s)|^2_{1.5}ds+\int_0^t |h (s)|^2_{2.5}ds\leq  C \,\left( \max_{0\leq s\leq t}|h(s)|^2_{2}+\int_0^t\|v(s)\|_{2,\pm}^2ds\right)\,.
\end{equation}
\end{lemma}
Note that Lemma \ref{parabolic} implies that the energy function $E(t)$ is continuous.

\subsection{Pressure estimates}
Using (\ref{HS_ALE}a) and (\ref{HS_ALE}b), $q^\pm$ solves
\begin{alignat}{2}
-(\beta^\pm J^\pm (A^\pm)^j_i(A^\pm)^k_iq^\pm,_{k}),_{j} & =0  && \text{ in }\refdom^\pm \,,\label{elliptic0}\\
q^+& =0 && \text{ on }\Gamma\,, \label{elliptic1}\\
\jump{q}& =0 && \text{ on }\Gamma_{\operatorname{perm}}\,,\label{elliptic2}\\
\jump{\beta q,_kA^k_iJA_i^rN^r}& = -\jump{\beta}\delta^2_iJ A_i^rN^r && \text{ on }\Gamma_{\operatorname{perm}}\,,\label{elliptic3}\\
\beta^- q^-,_k(A^-)^k_iJ^-(A^-)_i^rN^r& =-\beta^- && \text{ on }\Gamma_{\operatorname{bot}}\,.\label{elliptic4}
\end{alignat}

We have that $A(0) A(0)^T$ is symmetric and positive definite: 
$$
[A(0) A(0)^T]^i_j   \xi _i\xi _j \ge \mathcal{L} | \xi |^2 ;
$$
consequently, due to (\ref{boundpsi}b),
$$
\|A_0A_0^T-A(t) A^T(t)\|_{L^\infty}\leq  C \sqrt{t}\sqrt{E(t)}  \,,
$$
and we see that for $t$ sufficiently small, 
$$
{\frac{\mathcal{L} }{2}} |\xi|^2\leq [A (\cdot, t ) A^T(\cdot , t)]^i_j\xi^i\xi^j\leq 2 \mathcal{L} |\xi|^2.
$$
Thus, $A(t)A^T(t)$ form a uniformly elliptic operator for $t$ on $[0,T]$, and elliptic estimates (following the same approach as in \cite{CGS} and \cite{ChSh2014}) lead to
\begin{equation}\label{Q1.5}
\|q\|_{2.5,\pm}\leq C\sqrt{E(t)}  \,,\;
\|v(t)\|_{1.5,\pm}\leq  C\sqrt{E(t)}.
\end{equation}

Furthermore, using the same argument as in \cite[Section 8.4.5]{CGS}, we obtain that
\begin{equation}\label{RTpreserved}
\|q(t)-q(0)\|_{2.25,\pm}\leq t^{1/8}\mathcal{P}(E(t)),
\end{equation}
and, 
$$
\|q,_2(t)-q,_2(0)\|_{L^\infty(\Gamma)}\leq C|q,_2(t)-q,_2(0)|_{0.75}\leq t^{1/8}\mathcal{P}(E(t)).
$$
As a consequence of the latter inequality, the Rayleigh-Taylor sign condition holds in $[0,T]$ for small enough $T$. Furthermore, using the Sobolev embedding theorem,
$$
\|\nabla q(t)-\nabla q(0)\|_{L^\infty}\leq t^{1/8}\mathcal{P}(E(t)),
$$
and a stronger version of the bootstrap assumption (\ref{bootstrap}d) also holds.

\subsection{The energy estimates}In this section we will perform the basic energy estimates. Integrals of lower-order terms will be denoted by $\mathcal{R}(t)$, meaning that
$$
\int_0^t \mathcal{R}(s)ds\leq \mathcal{M}_0+\sqrt{t}\mathcal{P}(E(t)).
$$

We take two horizontal derivatives of (\ref{HS_ALE}a), test against $Jv''$, and integrate by parts to find that
$$
\int_0^t\int_{\refdom^+\cup\refdom^-}\frac{J}{\beta}|v''|^2 dxds+\int_0^t\int_{\refdom^+\cup\refdom^-}J\left[A^k_i (q+\psi^2)'',_{k}+(A^k_i)'' (q+\psi^2),_{k}\right]v_i''dx ds+\int_0^t \mathcal{R}(s)ds=0.
$$
Due to the divergence free condition (\ref{HS_ALE}b) we obtain that
\begin{equation}\label{eqincom}
A^k_i(v^i)'',_{k}=-(A'')^k_iv^i,_k+\mathcal{R}(t).
\end{equation}
Thus, integrating by parts and using \eqref{eqincom} and the identities $JA^k_iN^k=\tilde{n}^i, JA^k_iN^k=\tilde{n}^i_{\operatorname{perm}}$ and $JA^k_iN^k=N^i$ on $\Gamma$, $\Gamma_{\operatorname{perm}}$ and $\Gamma_{\operatorname{bot}}$, respectively, we find that
\begin{align*}
I_1&=\int_0^t\int_{\refdom^+\cup\refdom^-}JA^k_i (q+\psi^2)'',_{k}v_i''dxds\\
&=\int_0^t\int_{\Gamma}\tilde{n}^i (q+\psi^2)''v_i''dx_1ds
-\int_0^t\int_{\Gamma_{\operatorname{perm}}}\jump{v_i''(q+\psi^2)''
\tilde{n}_{\operatorname{perm}}^i}dx_1ds\\
&\quad-\int_0^t\int_{\Gamma_{\operatorname{bot}}}N^i (q+\psi^2)''v_i''dx_1ds
-\int_0^t\int_{\refdom^+\cup\refdom^-}JA^k_i (q+\psi^2)''(v^i),_{k}''dxds\\
&=\int_0^t\int_{\refdom^+\cup\refdom^-}J(A'')^k_i (q+\psi^2)''(v^i),_{k}dxds+\int_0^t\int_{\Gamma}\tilde{n}^i h''v_i''dx_1ds\\
&\quad-\int_0^t\int_{\Gamma_{\operatorname{perm}}}(q^++f)''\jump{v_i''\tilde{n}_{\operatorname{perm}}^i}dx_1ds+\int_0^t \mathcal{R}(s)ds.
\end{align*}
The 2-D integral is now a lower-order term that can be estimated with a $L^2-L^4-L^4-L^\infty$ H\"older argument together with the Sobolev embedding theorem. Thus, we are left with the integrals on the boundaries $\Gamma$ and $\Gamma_{\operatorname{perm}}$. Due to the incompressibility condition, we have that $\jump{v_i\tilde{n}_{\operatorname{perm}}^i}=0$, so that
\begin{align*}
I_1&=\int_0^t\int_{\Gamma}h''h_t''dx_1ds-\int_0^t\int_{\Gamma}(\sqrt{g}n_i)'' h''v_idx_1ds
\\
&\quad+\int_0^t\int_{\Gamma_{\operatorname{perm}}}(q^++f)''\jump{v_1}f'''dx_1ds+\int_0^t \mathcal{R}(s)ds\\
&=\frac{1}{2}|h''|_0^2-\frac{1}{2}|h_0''|_0^2-\frac{1}{2}\int_0^t\int_{\Gamma}(h'')^2v_1'dx_1ds
+\int_0^t|q''\jump{v_1}|_{0.5}|f|_{2.5}ds\\
&\quad -\frac{1}{2}\int_{\Gamma_{\operatorname{perm}}}(f'')^2\jump{v_1'}dx_1ds+\int_0^t \mathcal{R}(s)ds\\
&\geq\frac{1}{2}|h''|_0^2-\frac{1}{2}|h_0''|_0^2-\sqrt{t}\mathcal{P}(E(t)),
\end{align*}
where we have used H\"{o}lder inequality, the trace theorem, \eqref{Q1.5} and the inequality
\begin{equation}\label{almostalgebra}
|fg|_{0.5}\leq C_\lambda|f|_{0.5}|g|_{0.5+\lambda}.
\end{equation}

The remaining high order term can be handled as follows: using \eqref{propA} and integrating by parts,
\begin{align*}
I_2&=\int_0^t\int_{\refdom^+\cup\refdom^-}J(A^k_i)'' (q+\psi^2),_{k}v_i''dx ds\\
&=-\int_0^t\int_{\refdom^+\cup\refdom^-}JA^k_r\psi^r,_{11j}A^j_i (q+\psi^2),_{k}v_i''dx ds+\int_0^t \mathcal{R}(s)ds\\
&=\int_0^t\int_{\refdom^+\cup\refdom^-}\psi^r,_{11}JA^j_i(A^k_r (q+\psi^2),_{k}v_i''),_jdx ds-\int_0^t\int_{\Gamma}\psi^r,_{11}JA^j_i(A^k_r (q+\psi^2),_{k}v_i'')N^jdx_1ds\\
&\quad+\int_0^t\int_{\Gamma_{\operatorname{perm}}}\jump{\psi^r,_{11}JA^j_i(A^k_r (q+\psi^2),_{k}v_i'')}N^jdx_1 ds\\
&\quad-\int_0^t\int_{\Gamma_{\operatorname{bot}}}\psi^r,_{11}JA^j_i(A^k_r (q+\psi^2),_{k}v_i'')N^jdx_1 ds+\int_0^t \mathcal{R}(s)ds\\
&=\int_0^t\int_{\refdom^+\cup\refdom^-}\psi^r,_{11}JA^j_i(A^k_r (q+\psi^2),_{k}v_i''),_jdx ds-\int_0^t\int_{\Gamma}h''\tilde{n}_iA^k_2 (q+\psi^2),_{k}v_i''dx_1ds\\
&\quad+\int_0^t\int_{\Gamma_{\operatorname{perm}}}f''\jump{\tilde{n}_{\operatorname{perm}}^iA^k_2 (q+\psi^2),_{k}v_i''}dx_1 ds+\int_0^t \mathcal{R}(s)ds.
\end{align*}
The integral in the bulk of the fluid can be estimated using \eqref{eqincom} so that
$$
\int_0^t\int_{\refdom^+\cup\refdom^-}\psi^r,_{11}JA^j_i(A^k_r (q+\psi^2),_{k}v_i''),_jdx ds\geq-\sqrt{t}\mathcal{P}(E(t)).
$$
The integral over $\Gamma_{\operatorname{perm}}$ can be estimated using the $H^{0.5}-H^{-0.5}$ duality and \eqref{almostalgebra} as follows:
\begin{align*}
\int_0^t\int_{\Gamma_{\operatorname{perm}}}f''\jump{\tilde{n}_{\operatorname{perm}}^iA^k_2 (q+\psi^2),_{k}v_i''}dx_1 ds&\geq -\int_0^tC|f|_{2.5}(1+|f|_{1.75})|A\nabla(q+\psi^2)|_{0.75}|v|_{1.5}ds\\
&\geq -\int_0^tC\|A\nabla(q+\psi^2)\|_{1.25,\pm}\|v\|_{2,\pm}ds\\
&\geq -\sqrt{t}\mathcal{P}(E(t)).
\end{align*}
Using that $A^k_2\psi^2,_k=\delta^2_2=1$,
\begin{align*}
I_2&\geq-\sqrt{t}\mathcal{P}(E(t))-\int_0^t\int_{\Gamma}h''\tilde{n}_iv_i''(A^1_2 (q+\psi^2),_{1}+A^2_2 (q+\psi^2),_{2})dx_1ds\\
&\geq-\sqrt{t}\mathcal{P}(E(t))-\int_0^t\int_{\Gamma}h''\tilde{n}_iv_i''(A^2_2 q,_{2}+1)dx_1ds\\
&\geq-\sqrt{t}\mathcal{P}(E(t))-\int_0^t\int_{\Gamma}h''(h_t''-h'''v_1)(J^{-1} q,_{2}+1)dx_1ds\\
&\geq-\sqrt{t}\mathcal{P}(E(t))-\int_0^t\int_{\Gamma}h''h_t''\left(\frac{q,_{2}(t)}{J(t)}-\frac{q,_{2}(0)}{J(0)}+\frac{q,_{2}(0)}{J(0)}+1\right)dx_1ds.
\end{align*}
From \eqref{RTpreserved}, 
$$
I_2\geq -\sqrt{t}\mathcal{P}(E(t))-\int_0^t\int_{\Gamma}h''h_t''\left(\frac{q,_{2}(0)}{J(0)}+1\right)dx_1ds.
$$
Thus,
\begin{equation}\label{ineqenergy}
\int_0^t\left\|\sqrt{\frac{J}{\beta}}v''\right\|_{0,\pm}^2 ds+\frac{1}{2} \left|\sqrt{\frac{-q,_2(0)}{J(0)}}h''(t)\right|_0^2\leq \frac{1}{2} \left|\sqrt{\frac{-q,_2(0)}{J(0)}}h''_0\right|_0^2+\sqrt{t}\mathcal{P}(E(t)).
\end{equation}

\subsection{Elliptic estimates via the Hodge decomposition} In this section,  we use the following
\begin{lemma}[\cite{ChSh2014}]\label{Hodge2}Let $\Omega$ be a domain with boundary $\partial\Omega$ of Sobolev class $H^{k+0.5}$, $k\geq2$. Let $\psi_0$ be a given smooth mapping and define $\curl_{\psi_0} v$ and $\div_{\psi_0} v$ as in \eqref{divpsi} and \eqref{curlpsi}, respectively. Then for $v \in H^k(\Omega) $, 
\begin{equation*}
\|v\|_{H^k(\Omega)} \le C \Big[\|v\|_{L^2(\Omega)} + \|\curl_{\psi_0} v\|_{H^{k-1}(\Omega)} + \|\div_{\psi_0} v\|_{H^{k-1}(\Omega)} + \|v'\cdot \texttt{n}\|_{H^{k-1.5}(\partial\Omega)}\Big]\,,
\end{equation*}
where $\texttt{n}=(\psi_0')^\perp/|\psi_0'|$.  
\end{lemma}
Since in each phase,  $\operatorname{curl} u=0$ and $\operatorname{div}u=0$, it follows (see \cite{CGS}, Section 8.4.8) that
\begin{align}
\int_0^t\|\text{curl}_{\psi_0} v\|_{1,\pm}^2dy&\leq\sqrt{t}\mathcal{P}(E(t)),\label{curl}\\
\int_0^t\|\text{div}_{\psi_0} v\|_{1,\pm}^2dy&\leq\sqrt{t}\mathcal{P}(E(t))\label{div}.
\end{align}
First, we want to use Lemma \ref{Hodge2} to obtain an estimate for $\|v'\|_{1,\pm}$. The only term that is delicate is the boundary term $|v''\cdot \emph{\texttt{n}}|_{-0.5}$. For that term we have the following
\begin{lemma}[Estimates for the normal trace of $v$]Given a smooth solution to the Muskat problem (\ref{HS_ALE}a-g), 
\begin{equation}\label{boundaryterm}
\int_0^t|v''\cdot \texttt{n}|_{H^{-0.5}(\partial\refdom^-\cup \partial\refdom^+)}^2ds
\leq C\int_0^t\|v''\|_{0,\pm}^2ds+\sqrt{t}\mathcal{P}(E(t))\,,
\end{equation}
where $\texttt{n}=(\psi_0')^\perp/|\psi_0'|$.
\end{lemma}
\begin{proof}
In order to estimate $|v''\cdot\emph{\texttt{n}}|_{-0.5}$ using the $H^{1/2}-H^{-1/2}$ duality, we consider a function $\phi\in H^1(\refdom^+\cup \refdom^-)$. Due to the trace theorem, we have that $\phi\in H^{0.5}(\Gamma\cup\Gamma_{\operatorname{perm}}\cup\Gamma_{bot})$. We define the following integrals:
$$
I_1=\int_{\Gamma_{\operatorname{perm}}}g^{-1/2}(v^-)''_i JA^k_iN^k \phi dx_1,
$$
$$
I_2=\int_{\Gamma_{\operatorname{perm}}}g^{-1/2}(v^+)''_i JA^k_iN^k \phi dx_1,
$$
and
$$
I_3=\int_{\Gamma}g^{-1/2}v''_i JA^k_iN^k \phi dx_1.
$$
Using the fact that
$$
(\psi'_i)^\perp=JA^k_iN^k
$$
together with $v_2=0$ on $\Gamma_{bot}$, we see that in order we have the appropriate estimate for $|v''\cdot \emph{\texttt{n}}|_{-0.5}$, it is enough to obtain good bounds for $|I_1|,|I_2|$ and $|I_3|$. To do that we use the divergence theorem and Darcy's law (\ref{HS_ALE}a). We compute
\begin{align*}
I_1&=\int_{\Gamma_{\operatorname{perm}}}g^{-1/2}(v^-)''_i JA^k_iN^k \phi dx_1\\
&=\int_{\Gamma_{\operatorname{perm}}}g^{-1/2}(v^-)''_i JA^k_iN^k \phi dx_1+\int_{\Gamma_{bot}}g^{-1/2}(v^-)''_i JA^k_iN^k \phi dx_1\\
&=\int_{\refdom^-}(g^{-1/2}v''_i JA^k_i\phi),_k dx\\
&=\int_{\refdom^-}g^{-1/2}v''_i JA^k_i\phi,_k dx+\int_{\refdom^-}g^{-1/2},_kv''_i JA^k_i\phi dx+\int_{\refdom^-}(v^i,_k)'' g^{-1/2}JA^k_i\phi dx\,.
\end{align*}
Integrating by parts, we obtain that
\begin{align*}
\int_{\refdom^-}(v^i,_k)'' g^{-1/2}JA^k_i\phi dx&=-\int_{\refdom^-}\text{div}_{\psi}v'(Jg^{-1/2}\phi)' dx-\int_{\refdom^-}(v^i,_k)'Jg^{-1/2}(A^k_i)'\phi dx\,.
\end{align*}
So, we find that
\begin{align}
I_1&=\int_{\refdom^-}g^{-1/2}v''_i JA^k_i\phi,_k dx+\int_{\refdom^-}g^{-1/2},_kv''_i JA^k_i\phi dx\nonumber\\
&\quad-\int_{\refdom^-}\text{div}_{\psi}v'(Jg^{-1/2}\phi)' dx-\int_{\refdom^-}(v^i,_k)'Jg^{-1/2}(A^k_i)'\phi dx\,.\label{equation1}
\end{align}
Integrating by parts and using Piola's identity, we have that
\begin{align}
-\int_{\refdom^-}(v^i,_k)'Jg^{-1/2}(A^k_i)'\phi dx
&=\int_{\refdom^-}(v^i)'J(A^k_i)'(g^{-1/2}\phi),_k dx\nonumber\\ &\quad-\int_{\Gamma_{\operatorname{perm}}}(v^i)'Jg^{-1/2}(A^k_i)'\phi N^k dx_1\nonumber\\
&\quad+\int_{\Gamma_{\operatorname{bot}}}(v^i)'Jg^{-1/2}(A^k_i)'\phi N^k dx_1.\label{equation2}
\end{align}
Substituting \eqref{equation2} into equation \eqref{equation1} and using the boundary condition (\ref{HS_ALE}g) we obtain that
\begin{align}
I_1&=\int_{\refdom^-}g^{-1/2}v''_i JA^k_i\phi,_k dx+\int_{\refdom^-}g^{-1/2},_kv''_i JA^k_i\phi dx\nonumber\\
&\quad-\int_{\refdom^-}\text{div}_{\psi}v'(Jg^{-1/2}\phi)' dx+\int_{\refdom^-}(v^i)'J(A^k_i)'(g^{-1/2}\phi),_k dx\nonumber\\ &\quad-\int_{\Gamma_{\operatorname{perm}}}(v^i)'Jg^{-1/2}(A^k_i)'\phi N^k dx_1\,.\label{equation3}
\end{align}
Thus, using H\"older inequality, the Sobolev embedding theorem and the trace theorem together with Lemma \ref{lower}, we obtain that
\begin{align}
|I_1|&\leq C\left(\|v''\|_{0,-}+\|\text{div}_\psi v'\|_{0,-}\right)\|\phi\|_{1,-}\nonumber\\
&\quad+C\|v\|_{1.5,-}\|\psi\|_{2.5,-}\left(1+\|\psi\|_{2.5,-}\right)\|\phi\|_{1,-}\nonumber\\
&\quad+C\|v\|_{1.75,-}\|\psi\|_{2.5,-}\|\phi\|_{1,-}\,.\label{gammaperm}
\end{align}
We can use the continuity of the normal component of the velocity through $\Gamma_{\operatorname{perm}}$
$$
\jump{v_i JA^k_i N^k }=0\text{ on }\Gamma_{\operatorname{perm}}\times[0,T] \;,
$$
to write
$$
\jump{v''_i JA^k_i N^k }=-\jump{v_i (JA^k_i N^k)'' }-2\jump{v'_i (JA^k_i N^k)' }\text{ on }\Gamma_{\operatorname{perm}}\times[0,T] \;.
$$
Thus, once that we have an estimate for $I_1$, we have that
\begin{align*}
I_2&=I_1-\int_{\Gamma_{\operatorname{perm}}}\jump{g^{-1/2}v'_i (JA^k_i)'N^k} \phi dx_1-\int_{\Gamma_{\operatorname{perm}}}\jump{g^{-1/2}v_i (JA^k_i)''N^k} \phi dx_1\\
&=I_1-\int_{\Gamma_{\operatorname{perm}}}\jump{g^{-1/2}v'_i (JA^k_i)'N^k} \phi dx_1-\int_{\Gamma_{\operatorname{perm}}}\jump{v_1} \frac{f'''}{\sqrt{1+(f')^2}} \phi dx_1.
\end{align*}
Then, we can obtain an estimate for $I_2$ using our previous estimate for $I_1$. Thus, using \eqref{almostalgebra} for the last term, we have that
\begin{align}
\left|I_2\right|&\leq C\left(\|v''\|_{0,-}+\|\text{div}_\psi v'\|_{0,-}\right)\|\phi\|_{1,-}\nonumber\\
&\quad+C\|v\|_{1.5,-}\|\psi\|_{2.5,-}\left(1+\|\psi\|_{2.5,-}\right)\|\phi\|_{1,-}\nonumber\\
&\quad+C\|v\|_{1.75,-}\|\psi\|_{2.5,-}\|\phi\|_{1,-}\nonumber\\
&\quad+C\|v\|_{1.25,\pm}|f|_{2.5}\|\phi\|_{1,\pm}\,.\label{gammaperm2}
\end{align}
Similarly, we compute
\begin{align*}
I^+&=\int_{\Gamma}g^{-1/2}v''_i JA^k_iN^k \phi dx_1+\int_{\Gamma_{\operatorname{perm}}}g^{-1/2}v''_i JA^k_iN^k \phi dx_1\\
&=\int_{\refdom^+}(g^{-1/2}v''_i JA^k_i\phi),_k dx_1,
\end{align*}
so, following the same steps as in the estimate \eqref{gammaperm}, we have that
\begin{align*}
\left|I^+\right|&\leq C\left(\|v''\|_{0,\pm}+\|\text{div}_\psi v'\|_{0,\pm}\right)\|\phi\|_{1,\pm}\\
&\quad+C\|v\|_{1.5,\pm}\|\psi\|_{2.5,\pm}\left(1+\|\psi\|_{2.5,\pm}\right)\|\phi\|_{1,\pm}\\
&\quad+C\|v\|_{1.75,\pm}\|\psi\|_{2.5,\pm}\|\phi\|_{1,\pm}.
\end{align*}
Then,
\begin{align}
\left|I_3\right|&\leq C\left(\|v''\|_{0,\pm}+\|\text{div}_\psi v'\|_{0,\pm}\right)\|\phi\|_{1,\pm}\nonumber\\
&\quad+C\|v\|_{1.5,\pm}\|\psi\|_{2.5,\pm}\left(1+\|\psi\|_{2.5,\pm}\right)\|\phi\|_{1,\pm}\nonumber\\
&\quad+C\|v\|_{1.75,\pm}\|\psi\|_{2.5,\pm}\|\phi\|_{1,\pm}\nonumber\\
&\quad+C\|v\|_{1.25,\pm}|f|_{2.5}\|\phi\|_{1,\pm}\,.\label{gamma}
\end{align}
Using duality, \eqref{gammaperm}, and that $\psi^-(t)=\psi^-(0)$,
\begin{align*}
|v''\cdot \emph{\texttt{n}}|_{H^{-0.5}(\partial\refdom^-)}&=\sup_{|\phi|_{0.5}\leq 1}\left|\int_{\Gamma_{\operatorname{perm}}}g^{-1/2}(v^-)''_i JA^k_iN^k \phi dx_1\right|\\
&\leq C\left(\|v''\|_{0,-}+\|\text{div}_\psi v'\|_{0,-}\right)\|\phi\|_{1,-}\nonumber\\
&\quad+C\|v\|_{1.5,-}\|\psi\|_{2.5,-}\left(1+\|\psi\|_{2.5,-}\right)\|\phi\|_{1,-}\nonumber\\
&\quad+C\|v\|_{1.75,-}\|\psi\|_{2.5,-}\|\phi\|_{1,-}\,.
\end{align*}
Integrating in time and using \eqref{ht1.5}, \eqref{Q1.5} and \eqref{div}, we have that
$$
\int_0^t|v''\cdot \emph{\texttt{n}}|_{H^{-0.5}(\partial\refdom^-)}^2ds\leq C\int_0^t\|v''\|_{0,\pm}^2ds+\sqrt{t}\mathcal{P}(E(t)).
$$

Similarly, using \eqref{gammaperm2} and \eqref{gamma}, we obtain that 
\begin{align*}
|v''\cdot n(t)|_{H^{-0.5}(\partial\refdom^+)}&=\sup_{|\phi|_{0.5}\leq 1}\left|\int_{\Gamma}g^{-1/2}v''_i JA^k_iN^k \phi dx_1\right|\\
&\quad +\sup_{|\phi|_{0.5}\leq 1}\left|\int_{\Gamma_{\operatorname{perm}}}g^{-1/2}(v^+)''_i JA^k_iN^k \phi dx_1\right|\\
&\leq C\left(\|v''\|_{0,\pm}+\|\text{div}_\psi v'\|_{0,\pm}\right)\|\phi\|_{1,\pm}\nonumber\\
&\quad+C\|v\|_{1.5,\pm}\|\psi\|_{2.5,\pm}\left(1+\|\psi\|_{2.5,\pm}\right)\|\phi\|_{1,\pm}\nonumber\\
&\quad+C\|v\|_{1.75,\pm}\|\psi\|_{2.5,\pm}\|\phi\|_{1,\pm}\nonumber\\
&\quad+C\|v\|_{1.25,\pm}|f|_{2.5}\|\phi\|_{1,\pm}\,.
\end{align*}
Using Lemma \ref{lower}, \eqref{div} and taking the lifespan $T$ small enough, we have that
$$
\|\text{div}_{\psi} v'\|_{0,\pm}=\|\text{div}_{\psi_0} v'\|_{0,-}+c\|A(t)-A(0)\|_{L^\infty}\|v'\|_{1,\pm}\leq \sqrt[4]{t}\mathcal{P}(E(t))+\sqrt[5]{t}\|v'\|_{1,\pm}.
$$
Consequently, we have that
\begin{align*}
\int_0^t|v''\cdot n(s)|_{H^{-0.5}(\partial\refdom^+)} ^2ds&\leq c\int_0^t\|v''\|_{0,\pm}^2ds+\sqrt{t}\mathcal{P}(E(t))\,.
\end{align*}
Finally, using \eqref{Q1.5}, \eqref{almostalgebra}, the Sobolev embedding theorem and trace theorem, we have that
\begin{align*}
\left|v''\cdot \left((\psi'(t))^\perp/|\psi'(t)|-(\psi'_0)^\perp/|\psi'_0|\right)\right|_{H^{-0.5}(\partial\refdom^+)}&\leq c\|v\|_{2,\pm}\|J(t)A(t)-J(0)A(0)\|_{1.25}\\
&\leq \sqrt{t}\mathcal{P}(E(t))\,.
\end{align*}
Collecting these estimates, we conclude that
$$
\int_0^t|v''\cdot \emph{\texttt{n}}|_{H^{-0.5}(\partial\refdom^+)}^2ds\leq c\int_0^t\|v''\|_{0,\pm}^2ds+\sqrt{t}\mathcal{P}(E(t)).
$$
\end{proof}
Making use of Lemma \ref{Hodge2} (for $v'$) and \eqref{ineqenergy}, we obtain that
\begin{align*}
\int_0^t\|v'(s)\|_{1,\pm}^2ds &\le C \int_0^t\|v''\|_{0,\pm}^2 + \|\curl_{\psi_0} v'\|_{0,\pm}^2 + \|\div_{\psi_0} v'\|_{0,\pm}^2ds \\
&\quad + C\int_0^t|v''\cdot \emph{\texttt{n}}|_{H^{-0.5}(\partial\refdom^+\cup \partial\refdom^-)}^2ds\\
&\leq C\int_0^t\|v''\|_{0,\pm}^2ds+\sqrt{t}\mathcal{P}(E(t))\\
&\leq \mathcal{M}_0+\sqrt{t}\mathcal{P}(E(t))\,,
\end{align*}
where $\mathcal{M}_0$ is a constant depending on the initial data. Equipped with this last estimate and using Lemma \ref{Hodge2} and trace theorem, we obtain that
\begin{align}
\int_0^t\|v(s)\|_{2,\pm}^2ds &\le C \int_0^t\|v\|_{0,\pm}^2 + \|\curl_{\psi_0} v\|_{1,\pm}^2 + \|\div_{\psi_0} v\|_{1,\pm}^2ds \nonumber\\
&\quad + C\int_0^t|v'\cdot\emph{\texttt{n}}|_{H^{0.5}(\partial\refdom^+)\cup \partial\refdom^-)}^2ds\nonumber\\
&\le C \int_0^t\|v\|_{0,\pm}^2 + \|\curl_{\psi_0} v\|_{1,\pm}^2 + \|\div_{\psi_0} v\|_{1,\pm}^2+\|v'\cdot \emph{\texttt{n}}\|_{1,\pm}^2ds\nonumber\\
&\leq \mathcal{M}_0+\sqrt{t}\mathcal{P}(E(t))\,.\label{vH2}
\end{align}

\subsection{Conclusion}
Collecting the estimates \eqref{lowv}, \eqref{ineqenergy}, \eqref{vH2} and using the lower bound for $J(t)$ and the Rayleigh-Taylor sign condition, we find that
\begin{equation}\label{eqenergy123}
E(t)\leq \mathcal{M}_0+\sqrt{t}\mathcal{Q}(E(t)).
\end{equation}
From 
$$
h\in L^2(0,T;H^{2.5}(\Gamma)),\;h_t\in L^2(0,T;H^{1.5}(\Gamma)),
$$
the energy $E(t)$ is continuous and this inequality implies the existence of a uniform time $T(h_0,f)$ such that
$$
E(t)\leq 2\mathcal{M}_0. 
$$
Estimates showing the uniqueness of the solution follows from standard energy methods and the detailed analysis shown in \cite{CGS}.

\section{Proof of Theorem \ref{localsmall}: Local well-posedness}\label{sec4}
Based on the smoothing argument in \cite{CoSh2007, CoSh2010} and following \cite{CGS}, for $0<\kappa,\epsilon\ll1$ small enough, we define
$$
\Omega^+_{\kappa,\epsilon}(0)=\{(x_1,x_2)\in \TT\times \RR,\;\; -1+f(x_1)<x_2<\mathcal{J}_\kappa\mathcal{J}_\kappa \mathcal{J}_\epsilon h(x_1,0)\}.
$$
Now, following Section \ref{sec2}, we can construct an $H^{2.5}-$class diffeomorphism 
$$
\psi^+_{\kappa,\epsilon}(0):\refdom^+ \to \Omega^+_{\kappa,\epsilon}(0),\;\psi^-:\refdom^- \to \Omega^-.
$$ 

We consider the so called $\epsilon\kappa-$problem: 
\begin{subequations}
\begin{alignat*}{2}
\frac{(v_{\kappa,\epsilon}^\pm)^i}{\beta^\pm}+(A_{\kappa,\epsilon}^\pm)^k_i(q_{\kappa,\epsilon}^\pm+\psi_{\kappa,\epsilon}^\pm\cdot e_2),_{k}&=0 \qquad&&\text{in}\quad\refdom^\pm\times[0,T]\,,\\
(A_{\kappa,\epsilon}^\pm)^k_i(v_{\kappa,\epsilon}^\pm)^i,_k&=0 \qquad&&\text{in}\quad\refdom^\pm\times[0,T]\,,\\
(h_{\kappa,\epsilon})_t(t)&= (v_{\kappa,\epsilon}^+)^i J_{\kappa,\epsilon}^+ (A_{\kappa,\epsilon}^+)^2_i \qquad&&\text{on}\quad\Gamma\times[0,T]\,,\\
h_{\kappa,\epsilon}&= \mathcal{J}_\epsilon h_0 \qquad&&\text{on}\quad\Gamma\times\{0\}\,,\\
q_{\kappa,\epsilon}^+&=0 &&\text{on}\quad\Gamma\times[0,T]\,,\\
\jump{q_{\kappa,\epsilon}}&=0 &&\text{on}\quad\Gamma_{\operatorname{perm}}\times[0,T]\,,\\
\jump{q_{\kappa,\epsilon},_k (A_{\kappa,\epsilon})^k_i J_{\kappa,\epsilon} (A_{\kappa,\epsilon})^2_i }&=-\Bigjump{\frac{1}{\beta}}v_{\kappa,\epsilon}^i J_{\kappa,\epsilon} (A_{\kappa,\epsilon})^2_i  &&\text{on}\quad\Gamma_{\operatorname{perm}}\times[0,T]\,\\
v_{\kappa,\epsilon}^-\cdot e_2&=0 &&\text{on}\quad\Gamma_{\operatorname{bot}}\times[0,T]\,,\\
\Delta \psi_{\kappa,\epsilon}^+(t)&= \Delta \psi_{\kappa,\epsilon}^+(0)  \quad&&\text{in}\quad \refdom^+\times[0,T]\,,\\
\psi^+_{\kappa,\epsilon}(t)&= (x_1,x_2)+\mathcal{J}_\kappa\mathcal{J}_\kappa h_{\kappa,\epsilon}(x_1,t)e_2 \quad &&\text{on}\quad \Gamma\times[0,T]\,,\\
\psi^+_{\kappa,\epsilon}(t)&= (x_1,x_2)+f(x_1)e_2 \quad &&\text{on}\quad \Gamma_{\operatorname{perm}}\times[0,T]\,.
\end{alignat*}
\end{subequations}

Note that the $\epsilon$-regularization affects only the initial interface $h_0$, while the $\kappa$ regularization appears also in the PDE system.

The construction of smooth approximate solutions can be achieved with a fixed point scheme. The detailed construction of solutions to this problem is given in \cite{CGS}. See also \cite{pinhomo} for a very different approach to the construction of solutions using the integral kernel method. 

Once we are equipped with a smooth approximate solution $h_{\kappa,\epsilon}(x_1,t)$, we have to obtain uniform estimates in $\epsilon$ and $\kappa$. These uniform estimates in $\epsilon$ and $\kappa$ will allow us to pass to the limit. However, we need to take the limits in the appropriate order; to be able to take the limit as $\kappa\rightarrow0$ we need to have a smooth initial data ($H^s$, $s>2.5$ is enough), so, we need $\epsilon>0$. However, the term requiring $\epsilon>0$ is not present when $\kappa=0$. Thus, we have to take first the limit as $\kappa\rightarrow0$ and then the limit as $\epsilon\rightarrow0$ (see \cite{CGS} for more details). We define
$$
E_{\kappa,\epsilon}(t)=\max_{0\leq s\leq t}|\mathcal{J}_\kappa h_{\kappa,\epsilon}(s)|^2_{2}+\int_0^t\|v_{\kappa,\epsilon}(s)\|_{2,\pm}^2+|\mathcal{J}_\kappa \mathcal{J}_\kappa h_{\kappa,\epsilon}(s)|^2_{2.5}ds
$$
and follow the estimates in in Section \ref{sec3}. We obtain the $\kappa$-uniform bound
$$
E_{\kappa,\epsilon}(t)\leq 2\mathcal{M}_{0,\epsilon}\;\forall\,0\leq t\leq T_\epsilon. 
$$

Passing to the limit in $\kappa$ we obtain an approximate solution $h_\epsilon(x_1,t)$. Now we define
$$
E_{\epsilon}(t)=\max_{0\leq s\leq t}|h_{\epsilon}(s)|^2_{2}+\int_0^t\|v_{\epsilon}(s)\|_{2,\pm}^2+|h_{\epsilon}(s)|^2_{2.5}ds
$$
and follow the estimates in in Section \ref{sec3}. Recalling that the $\epsilon$-regularization only affects the initial interface $h_0$, we obtain the $\epsilon$-uniform bound
$$
E_{\epsilon}(t)\leq 2\mathcal{M}_{0} \;\forall\,0\leq t\leq T. 
$$ 
Passing to the limit in $\epsilon$ we obtain a local strong solution to the one-phase Muskat problem with discontinuous permeability \eqref{laplacian}.  

\section{The Muskat problem in the semi-ALE formulation}\label{semiale}
We again use \eqref{ref_domain} and \eqref{ref_domain2}, respectively,  for our reference domains and boundaries. We let $N=e_2$ denote
the unit normal vector on $\Gamma$, $\Gamma_{\operatorname{perm}}$ and $\Gamma_{\operatorname{bot}}$. Due to Theorem \ref{localsmall}, there exists a local solution $(h,u,p)$ to the Muskat problem \eqref{laplacian}.

We define $\delta\psi^\pm$ as the solution to
\begin{subequations}\label{deltapsi+}
\begin{alignat}{2}
\Delta \delta\psi^+&= 0  \qquad&&\text{in}\quad \refdom^{+}\times[0,T]\,,\\
\delta\psi^+ &= h \qquad &&\text{on}\quad \Gamma\times[0,T]\,,\\
\delta\psi^+ &= f &&\text{on}\quad \Gamma_{\operatorname{perm}}\times[0,T] \,,
\end{alignat}
\end{subequations}
\begin{subequations}\label{deltapsi-}
\begin{alignat}{2}
\Delta \delta\psi^-&= 0  \qquad&&\text{in}\quad \refdom^{-}\times[0,T]\,,\\
\delta\psi^- &= f \qquad &&\text{on}\quad \Gamma_{\operatorname{perm}}\times[0,T]\,,\\
\delta\psi^- &= 0 &&\text{on}\quad \Gamma_{\operatorname{bot}}\times[0,T] \,,
\end{alignat}
\end{subequations}
and
\begin{equation}\label{psi}
\psi^\pm(x_1,x_2)=(x_1,x_2)+(0,\delta\psi^\pm)\qquad \text{in}\quad \refdom^{\pm}\times[0,T],
\end{equation}
For all $s\in \mathbb{R}$, elliptic estimates show that
\begin{equation}\label{psibound}
\|\psi-e\|_{s+1/2,\pm}=\|\delta\psi\|_{s+1/5,\pm}\leq c(|h|_s+|f|_s)\ll1,
\end{equation}
due to the smallness of the initial data $h_0$ and the function $f$. Thus, for $s=2$, due to  the Sobolev embedding and inverse function theorems, 
$\psi^\pm$ is a $H^{2.5}-$class diffeomorphism. We define $J^\pm= \det(\nabla\psi^\pm)$ and $A^\pm = (\nabla \psi^\pm)^{-1}$.
In particular, 
$$
J^\pm=1+\delta\psi^\pm_{,2},
$$
$$
A^\pm= (J^\pm)^{-1} \left[\begin{array}{cc}
(\psi^\pm)^2,_2 & - (\psi^\pm)^1,_2 \\
-(\psi^\pm)^2,_1 & (\psi^\pm)^1,_1
\end{array}
\right]=\frac{1}{1+\delta\psi^\pm,_{2}} \left[\begin{array}{cc}
1+\delta\psi^\pm,_{2} & 0 \\
-\delta\psi^\pm,_{1} & 1
\end{array}
\right]\,.
$$

We define our ALE variables $v=u\circ \psi$, $q=p\circ \psi$ as in section \ref{sec2}. The new variables $v,q$ solve the system \eqref{HS_ALE}. We define our new semi-ALE variables
$$
w^i=JA^i_jv^j,\; Q=q+x_2.
$$
In particular, using Piola's identity and the equality $JA^i_jN^i=\sqrt{g}n^j$ valid on $\Gamma_{bot},$ $\Gamma_{\operatorname{perm}}$ and $\Gamma$, we have that
$$
w^i,_i=JA^i_jv^j,_i=0,
$$
and
$$
w_2^-=JA^2_jv^-_j=\delta^j_2v^-_j=v^-_2 \text{ on }\Gamma_{bot}.
$$
We also have that
$$
w^jN^j=JA^j_kv^kN^j=\sqrt{g}n_k v_k,
$$
so, due to the incompressibility of the fluid,
$$
\jump{w^jN^j}=0\text{ on }\Gamma_{\operatorname{perm}}
$$
Thus, these new variables solve 
\begin{subequations}\label{HS_semi_ALE}
\begin{alignat}{2}
\frac{(w^\pm)^j}{\beta^\pm}+J^\pm(A^\pm)^{j}_i(A^\pm)^k_i(Q^\pm+\delta\psi^\pm),_{k}&=0 \qquad&&\text{in}\quad\refdom^\pm\times[0,T]\,,\\
\text{div}\, w^\pm&=0 \qquad&&\text{in}\quad\refdom^\pm\times[0,T]\,,\\
h_t&= w^+_2\qquad&&\text{on}\quad\Gamma\times[0,T]\,,\\
h&= h_0\qquad&&\text{on}\quad\Gamma\times\{0\}\,,\\
Q^+&=0 &&\text{on}\quad\Gamma\times[0,T]\,,\\
\jump{Q}&=0 &&\text{on}\quad\Gamma_{\operatorname{perm}}\times[0,T]\,,\\
\jump{(Q+\delta\psi),_k A^k_i J A^j_i N^j}&=-\Bigjump{\frac{1}{\beta}}w^+_j N^j &&\text{on}\quad\Gamma_{\operatorname{perm}}\times[0,T]\,\\
w^-_2&=0 &&\text{on}\quad\Gamma_{\operatorname{bot}}\times[0,T]\,.
\end{alignat}
\end{subequations}
Equivalently,
\begin{subequations}\label{HS_semi_ALE2}
\begin{alignat}{2}
\frac{w^\pm}{\beta^\pm}+\nabla(Q^\pm+\delta\psi^\pm)&=\left(\text{Id}-\frac{(\nabla\psi^\pm)^T
\nabla\psi^\pm}{J^\pm}\right)\frac{w^\pm}{\beta^\pm} \qquad&&\text{in}\quad\refdom^\pm\times[0,T]\,,\\
\text{div}\, w^\pm&=0 \qquad&&\text{in}\quad\refdom^\pm\times[0,T]\,,\\
h_t&= w^+_2\qquad&&\text{on}\quad\Gamma\times[0,T]\,,\\
h&= h_0\qquad&&\text{on}\quad\Gamma\times\{0\}\,,\\
Q^+&=0 &&\text{on}\quad\Gamma\times[0,T]\,,\\
\jump{Q}&=0 &&\text{on}\quad\Gamma_{\operatorname{perm}}\times[0,T]\,,\\
\jump{(Q+\delta\psi),_k A^k_i J A^j_i N^j}&=-\Bigjump{\frac{1}{\beta}}w^+_j N^j &&\text{on}\quad\Gamma_{\operatorname{perm}}\times[0,T]\,\\
w^-_2&=0 &&\text{on}\quad\Gamma_{\operatorname{bot}}\times[0,T]\,.
\end{alignat}
\end{subequations}
Using the particular form of $\nabla\psi^\pm=\nabla(x_1,x_2+\delta\psi^\pm)\,,$
we have that
\begin{equation}\label{eq1}
\left(\text{Id}-\frac{(\nabla\psi)^T\nabla\psi^\pm }{J^\pm}\right)\frac{w^\pm}{\beta^\pm}=\left(\begin{array}{cc}\delta\psi^\pm,_{2}-(\delta\psi^\pm,_{1})^2 & -\delta\psi^\pm,_{1}(1+\delta\psi^\pm,_{2})\\
-\delta\psi^\pm,_{1}(1+\delta\psi^\pm,_{2}) & -\delta\psi^\pm,_{2}(1+\delta\psi^\pm,_{2})\end{array}\right)\frac{w^\pm}{\beta^\pm J^\pm},
\end{equation}
and we see that the right hand side of (\ref{HS_semi_ALE2}a) contains all the non-linear terms.

\section{Proof of Theorem \ref{globalsmall}: Global well-posedness when $f=0$}\label{secglobal}
We define the energy $\mathscr{E}(t)$ and energy dissipation $\mathscr{D}(t)$ as follows
\begin{equation}\label{energyE}
\mathscr{E}(t)=|h''(t)|_0^2,\;\mathscr{D}(t)=\|w''(t)\|_{0,\pm}^2.
\end{equation}
As $h(\cdot,t)$ has zero mean, the Poincar\'e inequality shows that 
$$
|h|_n\equiv |h^{n)}|_0,\;n\in \mathbb{Z}^+.
$$

By hypothesis, the initial data $h_0$ satisfies the smallness condition
\begin{equation}\label{smallnesscondition}
|h_0|_2<\mathscr{C},
\end{equation}
for $\mathscr{C}$ a small enough constant. 

Note that due to Theorem \ref{localsmall}, there exists a time $T$ such that
$$
\mathscr{E}(t)<2\mathscr{E}(0)<2\mathscr{C}^2\ll1.
$$

Let us sketch the proof of the theorem. Our goal is to prove that, for initial data satisfying the smallness condition \eqref{smallnesscondition}, the system remains in the Rayleigh-Taylor stable regime and verifies the following estimates 
$$
\sup_{0\leq t}|h(t)|_2\leq |h_0|_2<\mathscr{C},\;\int_0^t\|w(s)\|_{2,\pm}^2ds\leq C_1\mathscr{C}^2,
$$
where $C_1$ is a time independent constant. Then, a standard continuation argument for ODE in Banach spaces implies that the local solution provided by Theorem \ref{localsmall} is, in fact, a global-in-time solution.

This goal is achieved in several steps. First we prove that, for $\mathscr{C}$ small enough, the system remains in the Rayleigh-Taylor stable case and verifies the estimate
\begin{equation}
\frac{d}{dt}\mathscr{E}+\mathscr{D}\leq \mathscr{D}\sqrt{\mathscr{E}}\mathcal{P}(\mathscr{E})\quad \forall 0\leq t\leq T\label{sketch}.
\end{equation}
This inequality implies the decay of $\mathscr{E}$ for small enough initial data; however, to obtain the rate of decay, we need to relate the energy $\mathscr{E}$ with the energy dissipation $\mathscr{D}$. To do that we establish the estimate
$$
|h''|_{0.5}\leq C\left(\|w\|_{2,\pm}|h|_{1.75}+\|w''\|_{0,\pm}\right).
$$
This estimate relies on Darcy's law. Using the smallness condition \eqref{smallnesscondition} and the Hodge decomposition elliptic estimate (see Lemma \ref{Hodge2}), we prove that tangential derivatives of the velocity are enough to control the full $H^2$ norm of the velocity field:
$$
\|w\|_{2,\pm} \leq C \Big[\|w\|_{1.5,\pm} + \|w''\|_{0,\pm}\Big]\,.
$$
Finally, using the smallness of $\mathscr{C}$, we can relate the energy $\mathscr{E}$ with the energy dissipation $\mathscr{D}$ as follows
$$
\mathscr{E}\leq |h''|_{0.5}\leq C\|w''\|_{0,\pm}=C\mathscr{D}.
$$
Thus, using the smallness of $\mathscr{C}$, the previous energy estimate \eqref{sketch} is equivalent to
\begin{align}\label{decay3}
\gamma\mathscr{E}+\frac{d}{dt}\mathscr{E}
&\leq 0.
\end{align}
for a certain $\gamma>0$. Note that due to the definition of $\mathscr{E}$, we have that \eqref{decay3} implies
$$
|h''(t)|_0\leq |h_0''|_0e^{-\gamma t/2}.
$$
We also obtain that
$$
\int_0^t\|w(s)\|_{2,\pm}^2ds\leq C(\mathscr{E}(0),\beta^\pm).
$$

\subsection{Pressure estimates}
Recall that, as $\delta\psi^-=0$ on $\Gamma_{bot},$ we have that $A$ verifies
$$
(A^-)^{2}_i((A^-)^1_i=0,\; J(A^-)^{2}_i((A^-)^2_i=(A^-)^2_2=\frac{1}{J^-}=\frac{1}{1+\delta\psi^-,_2}\text{ on }\Gamma_{bot}.
$$
Note also that (\ref{HS_semi_ALE2}f) is equivalent to
$$
\jump{\beta(Q+\delta\psi),_k A^k_i J A^j_i N^j}=0 \text{ on }\Gamma_{\operatorname{perm}}\times[0,T],
$$
thus, multiplying (\ref{HS_semi_ALE2}a) by $\beta^\pm$ and using the divergence free condition (\ref{HS_semi_ALE2}b), the modified pressure $Q$ solves
\begin{subequations}\label{Q}
\begin{alignat}{2}
\beta^\pm J^\pm(A^\pm)^{j}_i((A^\pm)^k_i(Q^\pm+\delta\psi^\pm),_{k}),_j&=0 \qquad&&\text{in}\quad\refdom^\pm\times[0,T]\,,\\
Q^+&=0 &&\text{on}\quad\Gamma\times[0,T]\,,\\
\jump{Q}&=0 &&\text{on}\quad\Gamma_{\operatorname{perm}}\times[0,T]\,,\\
\jump{\beta(Q+\delta\psi),_k A^k_i J A^j_i N^j}&=0 &&\text{on}\quad\Gamma_{\operatorname{perm}}\times[0,T]\,\\
\beta^-(A^-)^2_2(Q^-+\delta\psi^-),_{2}&=0 &&\text{on}\quad\Gamma_{\operatorname{bot}}\times[0,T]\,.
\end{alignat}
\end{subequations}

Equivalently, using \eqref{deltapsi+} and \eqref{deltapsi-}, \eqref{Q} can be written as
\begin{subequations}\label{Qv2}
\begin{alignat}{2}
\beta\Delta Q&=\beta\text{div}\,\left[(\text{Id}- JAA^T)\nabla(Q+\delta\psi)\right] \qquad&&\text{in}\quad\refdom^\pm\times[0,T]\,,\\
Q^+&=0 &&\text{on}\quad\Gamma\times[0,T]\,,\\
\jump{Q}&=0 &&\text{on}\quad\Gamma_{\operatorname{perm}}\times[0,T]\,,\\
\jump{\beta Q,_2}&=\jump{\beta Q,_k (\delta^k_2-A^k_i J A^2_i)}-\jump{\beta\delta\psi,_k A^k_i J A^2_i } &&\text{on}\quad\Gamma_{\operatorname{perm}}\times[0,T]\,\\
\beta^-Q^-,_2&=\beta^-(Q^-,_2-1)\frac{\delta\psi^-,_{2}}{1+\delta\psi^-,_2} &&\text{on}\quad\Gamma_{\operatorname{bot}}\times[0,T]\,,
\end{alignat}
\end{subequations}
where
$$
\text{Id}-JAA^T=\left[\begin{array}{cc}
-\delta\psi^\pm,_{2} & \delta\psi^\pm,_1 \\
\delta\psi^\pm,_{1} & \frac{\delta\psi^\pm,_2-(\delta\psi^\pm,_1)^2}{1+\delta\psi,_2^\pm}
\end{array}
\right],
$$
and, using (\ref{deltapsi+}c) and (\ref{deltapsi-}b), 
$$
A^1_i J A^2_i=0,\;A^2_i J A^2_i=\frac{1}{1+\delta\psi,_2}\text{ on }\Gamma_{\operatorname{perm}}.
$$

Elliptic estimates and trace theorem then show that
\begin{align*}
\|\nabla Q\|_{1.5,\pm}&\leq C\bigg{(}\|(\text{Id}- JAA^T)\nabla(Q+\delta\psi)\|_{1.5,\pm}+\left|\jump{\beta Q,_k (\delta^k_2-A^k_i J A^2_i)}\right|_{1}\\
&\quad +\left|\jump{\beta\delta\psi,_k A^k_i J A^2_i }\right|_1 +\left|\beta^-(Q^-,_2-1)\frac{\delta\psi^-,_{2}}{1+\delta\psi^-,_2}\right|_1\bigg{)}\\
&\leq C\bigg{(}\|\text{Id}- JAA^T\|_{L^\infty}\|\nabla Q\|_{1.5,\pm}+\|\text{Id}- JAA^T\|_{1.5,\pm}\|\nabla Q\|_{L^\infty}\\
&\quad+\|\text{Id}- JAA^T\|_{L^\infty}\|\nabla \delta\psi\|_{1.5,\pm}+\|\text{Id}- JAA^T\|_{1.5,\pm}\|\nabla \delta\psi\|_{L^\infty}\\
&\quad +\left\|\frac{\delta\psi^-,_{2}}{1+\delta\psi^-,_2}\right\|_{L^\infty}\|\nabla Q\|_{1.5,\pm}+\left\|\frac{\delta\psi,_{2}}{1+\delta\psi,_2}\right\|_{1.5,\pm}\|\nabla Q\|_{L^\infty}\\
&\quad +\left\|\frac{\delta\psi,_{2}}{1+\delta\psi,_2}\right\|_{1.5,\pm}\bigg{)}.
\end{align*}
Thus, using \eqref{psibound}, we have that
\begin{equation}\label{Q2.5}
\|\nabla Q\|_{1.5,\pm}\leq C|h|_2^2\mathcal{P}(|h|_2^2).
\end{equation}

Using (\ref{HS_semi_ALE}a), we obtain that
\begin{equation}\label{w1.5}
\|w\|_{1.5,\pm}\leq C|h|_2^2\mathcal{P}(|h|_2^2).
\end{equation}

\subsection{The Rayleigh-Taylor stability condition}
In the previous ALE variables $(v,q)$, the Rayleigh-Taylor stability condition \eqref{RT} reads
$$
- A^k_jq,_kJA_j^iN^i= -A^2_jq,_2JA_j^2> 0 \ \text{ on } \ \Gamma\,.
$$
In our semi-ALE modified pressure, we have that the Rayleigh-Taylor stability condition is equivalent to
$$
-A^2_j(Q,_2-1)JA_j^2> 0 \ \text{ on } \ \Gamma\,,
$$
or, using (\ref{HS_semi_ALE}a,c),
$$
h_t+JA^2_iA^2_i(1+\delta\psi^+,_2)+JA^2_iA^1_ih'+> 0 \ \text{ on } \ \Gamma\,.
$$
Thus, the Rayleigh-Taylor stability condition is equivalent to
$$
w_2^++1> 0 \ \text{ on } \ \Gamma\,.
$$
Then, we see that when $\|w_2\|_{L^\infty}\ll1$, the Rayleigh-Taylor stability condition holds. Thus, using the Sobolev embedding theorem together with \eqref{w1.5}, we have that
\begin{equation*}
\|w_2^+\|_{L^\infty(\Gamma)}\leq C\|w \|_{1.5,\pm}\leq C|h|_2^2\mathcal{P}(|h|_2^2)\ll1.
\end{equation*}

\subsection{The estimates in $L^2(0,T;H^{2.5}(\Gamma))$}
Recalling (\ref{HS_semi_ALE2}a,d), we have that
\begin{equation*}
w_1^++\beta^+h'=(\delta\psi^+,_{2}-(h')^2)\frac{w_1^+}{1+\delta\psi^+,_2}-h'w_2^+\text{ on }\Gamma.
\end{equation*}
Thus,
\begin{equation*}
\beta^+h'=-(1+(h')^2)\frac{w_1^+}{1+\delta\psi^+,_2}-h'w_2^+\text{ on }\Gamma,
\end{equation*}
and
\begin{equation*}
-\beta^+h''=(1+(h')^2)\frac{(w_1^+),_1}{1+\delta\psi^+,_2}+h'(w_2^+),_1+\frac{2h'h''w_1^+}{1+\delta\psi^+,_2}+h''w_2^+-\frac{1+(h')^2w_1^+}{(1+\delta\psi^+,_2)^2}\delta\psi,_{12}
\end{equation*}
Using the smallness of $|h|_2$, \eqref{almostalgebra} and estimate \eqref{w1.5}, we then estimate
\begin{align*}
|h''|_{0.5}&\leq C\left(|w'_1|_{0.5}+|w'_2|_{0.5}\right)
\end{align*}
The following Lemma is an immediate consequence of the standard normal trace theorem (see Temam \cite{temam2001navier}):
\begin{lemma}\label{normaltrace} Suppose that $v'\in L^2(\Omega)$ with $\text{div}v\in L^2(\Omega)$. Then $v'\cdot N \in H^ {-\frac{1}{2}} (\partial\Omega)$
and
\begin{equation*}
\|v' \cdot N\|_{H^{-1/2}(\partial\Omega)}\leq C\left(\|v' \|_{L^2(\Omega)}+\| \operatorname{div} v \|_ { L^2(\Omega) }  \right).
\end{equation*}
\end{lemma}
Using Lemma \ref{normaltrace} for $v=w'$, we obtain that
\begin{equation*}
|w'' _2|_{-0.5}=|w'' \cdot N|_{-0.5}\leq C\|w'' \|_{0,\pm}.
\end{equation*}
Using Lemma \ref{normaltrace} for $v=(w^\perp)'$,
\begin{equation*}
|w_1''|_{-0.5}=|(w^\perp)'' \cdot N|_{-0.5}\leq C\left(\|w'' \|_{0,\pm}+\| \operatorname{div} w,_1^\perp \|_ {0,\pm }  \right)=C\left(\|w'' \|_{0,\pm}+\| \operatorname{curl} w' \|_ {0,\pm }  \right).
\end{equation*}
Using that $\operatorname{curl} u=0$, we find that (see Cheng, Granero-Belinch\'on \& Shkoller \cite[Section 5.1.7]{CGS})
\begin{equation}\label{curlw}
\|\curl w\|_{1,\pm}\leq C\|w\|_{2,\pm}|h|_{1.75}+C|h|_{2.5}|h|_{1.75}+C|h|_{2}\|w\|_{1.5,\pm}^2.
\end{equation}
Thus, using the Poincar\'e inequality together with \eqref{w1.5} and the smallness of $|h|_2$, we find that
\begin{align}\label{h2.5global}
|h''|_{0.5}&\leq C\left(\|w\|_{2,\pm}|h|_{1.75}+\|w''\|_{0,\pm}\right).
\end{align}
\subsection{Hodge decomposition elliptic estimates} Using Lemma \ref{Hodge2} (with $\psi_0=(x_1,x_2)$) we have that
$$
\|w\|_{2,\pm} \le C \Big[\|w\|_{0,\pm} + \|\curl w\|_{1,\pm} + \|\div w\|_{1,\pm} + |w_2|_{1.5}\Big]\,.
$$

As a consequence of \eqref{curlw}, \eqref{h2.5global}, the Poincar\'e inequality and Lemma \ref{normaltrace}, we find that
\begin{align*}
\|w\|_{2,\pm} &\leq C \Big[\|w\|_{1.5,\pm} + \|\curl w\|_{1,\pm} + |w''\cdot N|_{-0.5}\Big]\\
&\leq C \Big[\|w\|_{1.5,\pm} + \|w\|_{2,\pm}|h|_{1.75}+|h|_{2.5}|h|_{1.75}+|h|_{2}\|w\|_{1.5,\pm}^2 + \|w''\|_{0,\pm}\Big]\\
&\leq C \Big[\|w\|_{1.5,\pm} + \|w\|_{2,\pm}|h|_{1.75}+|h|_{2}\|w\|_{2,\pm}\|w\|_{1.5,\pm} + \|w''\|_{0,\pm}\Big]\,.
\end{align*}
As a consequence of the smallness of $|h|_2$ and \eqref{w1.5}, we have that
\begin{align}\label{w2}
\|w\|_{2,\pm} &\leq C \Big[\|w\|_{1.5,\pm} + \|w''\|_{0,\pm}\Big]\,.
\end{align}
Substituting this last inequality into \eqref{h2.5global} together with $|h|_2\ll1$ and \eqref{w1.5}, we obtain that
\begin{align}
|h''|_{0.5}&\leq C\left(\|w\|_{1.5,\pm}|h|_{1.75}+\|w''\|_{0,\pm}\right)\nonumber\\
&\leq C\|w''\|_{0,\pm}\label{h2.5global2}
\end{align}

\subsection{The energy estimates}
The goal in this section is to prove that, the solution verifies the following bound
$$
\frac{\mathscr{D}}{\max\{\beta^+,\beta^-\}}+\frac{1}{2}\frac{d}{dt}\mathscr{E}\leq \mathscr{D}\mathcal{Q}(\mathscr{E}),
$$
where $\mathcal{Q}$ is a polynomial such that $\mathcal{Q}(0)=0$. Then, for small enough initial data, we have that
$$
C\mathscr{D}+\frac{d}{dt}\mathscr{E}\leq0,
$$
and we conclude the decay of $\mathscr{E}$. To obtain the exponential rate of decay in Theorem \ref{globalsmall}, we will invoke \eqref{h2.5global2} and Poincar\'e inequality.

We take two tangential derivatives of (\ref{HS_semi_ALE2}a) and test against $w''$. We obtain that
\begin{equation}\label{energyineqw}
\int_{\refdom^+\cup\refdom^-}\frac{|w''|^2}{\beta}dx+\int_{\refdom^+\cup\refdom^-}(Q+\delta\psi)'',_iw''_idx=\int_{\refdom^+\cup\refdom^-}\left[\left(\text{Id}
-\frac{(\nabla\psi)^T
\nabla\psi}{J}\right)\frac{w}{\beta}\right]''w''dx.
\end{equation}
Integrating by parts and using \eqref{HS_semi_ALE2}, we have that
\begin{align*}
\int_{\refdom^+\cup\refdom^-}(Q+\delta\psi)'',_iw''_idx&=\int_{\Gamma}(Q+\delta\psi)''w''_iN^idx_1-\int_{\Gamma_{\operatorname{perm}}}\jump{(Q+\delta\psi)''w''_iN^i}dx_1\\
&\quad -\int_{\Gamma_{bot}}(Q+\delta\psi)''(w''_i)N^idx_1\\
&=\int_{\Gamma}h''h_t''dx_1.
\end{align*}
Thus due to \eqref{eq1}, \eqref{psibound}, \eqref{w1.5}, \eqref{h2.5global2}, H\"older's inequality and the Sobolev embedding theorem, \eqref{energyineqw} is equivalent to
\begin{align*}
\frac{\mathscr{D}}{\max\{\beta^+,\beta^-\}}+\frac{1}{2}\frac{d}{dt}\mathscr{E}&\leq \int_{\refdom^+\cup\refdom^-}\left[\left(\text{Id}
-\frac{(\nabla\psi)^T
\nabla\psi}{J}\right)\frac{w}{\beta}\right]''w''dx\\
&\leq \mathscr{D}\sqrt{\mathscr{E}}\mathcal{P}(\mathscr{E})+\bigg{|}\int_{\refdom^+\cup\refdom^-}\left(\text{Id}
-\frac{(\nabla\psi)^T
\nabla\psi}{J}\right)''\frac{w}{\beta}w''dx\bigg{|}\\
&\leq\mathscr{D}\sqrt{\mathscr{E}}\mathcal{P}(\mathscr{E})+\|\nabla\delta\psi\|_{2,\pm}\mathcal{P}(\mathscr{E})\|w\|_{L^\infty}\sqrt{\mathscr{D}}\\
&\leq \mathscr{D}\sqrt{\mathscr{E}}\mathcal{P}(\mathscr{E})+|h|_{2.5}\sqrt{\mathscr{E}}\mathcal{P}(\mathscr{E})\sqrt{\mathscr{D}}\\
&\leq \mathscr{D}\sqrt{\mathscr{E}}\mathcal{P}(\mathscr{E}),
\end{align*}
where we have used Young's inequality. We note that due to Theorem \ref{localsmall}, we have that
$$
\mathscr{E}(t)\leq 2\mathscr{E}(0) \text{ for all }0\leq t\leq T^*,
$$
thus
$$
\mathcal{P}(\mathscr{E}(t))\leq C(h_0)\text{ for all }0\leq t\leq T^*.
$$
Taking $\mathscr{C}$ small enough, we obtain the inequality
\begin{align*}
\epsilon\mathscr{D}+\frac{1}{2}\frac{d}{dt}\mathscr{E}&\leq 0,
\end{align*}
for certain $\epsilon=\epsilon(h_0,\beta)$. Thus, using Poincar\'e inequality and \eqref{h2.5global2}
\begin{align*}
\gamma\mathscr{E}+\frac{d}{dt}\mathscr{E}
&\leq 0,
\end{align*}
for certain $\gamma=\gamma(h_0,\beta)$. 
\begin{align}\label{decay}
|h''(t)|_0\leq |h''_0|_0e^{-\gamma t/2}.
\end{align}
We also obtain the bound
\begin{align}\label{boundw2}
\int_0^t\mathscr{D}(s)ds\leq C(h_0,\beta).
\end{align}

\section*{Acknowledgements}
RGB is funded by the Labex MILYON and the Grant MTM2014-59488-P from the Ministerio de Econom\'ia y Competitividad (MINECO, Spain). SS was supported by the National Science Foundation under grant  DMS-1301380.

\bibliographystyle{abbrv}

\end{document}